\documentclass[11pt,a4paper,reqno]{amsart}%article}%
\usepackage{amsfonts}
\usepackage{amsmath,amssymb,amsthm,amsxtra}
\usepackage{float}
\usepackage[colorlinks, linkcolor=blue!50,anchorcolor=Periwinkle,
    citecolor=blue!72,urlcolor=cyan, bookmarksopen,bookmarksdepth=2]{hyperref}
\usepackage[usenames,dvipsnames]{xcolor}
\usepackage{enumitem}
\usepackage{geometry,array}
\usepackage{graphicx}
\usepackage{subfigure}
\usepackage{bookmark}
\usepackage{tikz}
\usepackage{url}

\usetikzlibrary{matrix,positioning,decorations.markings,arrows,decorations.pathmorphing,
    backgrounds,fit,positioning,shapes.symbols,chains,shadings,fadings,calc}
\tikzset{->-/.style={decoration={  markings,  mark=at position #1 with
    {\arrow{>}}},postaction={decorate}}}
\tikzset{-<-/.style={decoration={  markings,  mark=at position #1 with
    {\arrow{<}}},postaction={decorate}}}

\usepackage[all]{xy}
\def\XX{\mathbb{X}}
\def\Zq{R}
\def\xx{\mathbb{X}}
\def\OStab{\QStab^{\oplus}}
\def\CStab{\QStab^{*}}

\def\QStab{\operatorname{QStab}}

\newcommand{\Note}[1]{\textcolor{black}{\texttt{#1}}}

\theoremstyle{plain}
\newtheorem{theorem}{Theorem}[section]
\newtheorem*{thma}{Theorem~A}
\newtheorem*{thmb}{Theorem~B}

\newtheorem{lemma}[theorem]{Lemma}
\newtheorem{corollary}[theorem]{Corollary}
\newtheorem{proposition}[theorem]{Proposition}
\newtheorem{conjecture}[theorem]{Conjecture}
\theoremstyle{definition}
\newtheorem{definition}[theorem]{Definition}

\newtheorem{example}[theorem]{Example}

\newtheorem{remark}[theorem]{Remark}

\numberwithin{equation}{section}
\newtheorem{assumption}[theorem]{Assumption}
\newtheorem{construction}[theorem]{Construction}

\newtheorem*{con}{Convention}
\def\hua{\mathcal}
\def\hh{\mathcal}
\def\ha{\hh{A}}
\def\kong{\mathbb}
\def\<{\langle}
\def\>{\rangle}
\def\ZZ{\mathbb{Z}}

\def\R{\mathbb{R}}
\def\RR{\R}

\def\Aut{\operatorname{Aut}}

\def\Sim{\operatorname{Sim}}
\def\Hom{\operatorname{Hom}}
\def\hom{{\hh{H}}om}

\def\Ext{\operatorname{Ext}}

\def\Stab{\operatorname{Stab}}

\def\Stap{\operatorname{Stab}^\circ}
\def\diff{\operatorname{d}}
\def\Br{\operatorname{Br}}

\def\rank{\operatorname{rank}}

\def\deg{\operatorname{deg}}
\newcommand{\h}{\hh{H}}            %heart

\newcommand{\ns}{\widehat{\sigma}}
\newcommand{\nz}{\widehat{Z}}
\newcommand{\np}{\widehat{\hh{P}}}
\renewcommand{\k}{\mathbf{k}}

\renewcommand{\Re}{\operatorname{Re}}

\newcommand{\shift}[1]{\operatorname{\Sigma}_{#1}}
\newcommand{\D}{\operatorname{\hh{D}}}
\newcommand{\C}{\operatorname{\hh{C}}}

\newcommand{\per}{\operatorname{per}}

\def\QQX{\qq{\XX}}
\def\surf{\mathbf{S}}                       %FST's surface
\newcommand{\ST}{\operatorname{ST}}        %spherical twists
\def\coh{\operatorname{Coh}}

\def\Fuk{\operatorname{Fuk}}

\def\RHom{\operatorname{RHom}}

\def\add{\operatorname{add}}
\newcommand{\Quad}{\operatorname{Quad}}

\newcommand{\TFuk}{\operatorname{TFuk}}

\def\filt{\operatorname{filt}}
\def\Grot{\operatorname{\mathrm{K}}}
\def\gldim{\operatorname{gldim}}

\def\sadd{\sigma_{\oplus}}
\def\sext{\sigma_{*}}
\def\padd{\hh{P}_{\oplus}}
\def\pext{\hh{P}_{*}}
\def\Coh{\mathrm{Coh}}

\DeclareMathOperator{\Per}{per}
\def\bi{\mathbf{i}}
\def\lto{\longrightarrow}
\def\iso{\xrightarrow{\sim}}
\def\sli{\mathcal{P}}
\newcommand{\norm}[1]{\lVert #1 \rVert}
\def\CC{\mathbb{C}}

\def\hs{\mathfrak{h}}
\def\reg{\mathrm{reg}}
\def\alp{\alpha}

\def\bK{\mathbf{k}}

\def\TK{\mathrm{T} \mathcal{K}_X^{-1}}
\def\OO{\mathcal{O}}
\def\K{\mathcal{K}}
\def\E{\mathcal{E}}

\def\reg{\mathrm{reg}}

\def\Im{\operatorname{Im}}
\def\GL{\operatorname{GL}}
\def\DQ{\D_\infty(Q)}
\def\DXQ{\D_\XX(Q)}
\def\DNQ{\D_N(Q)}
\newcommand{\qq}[1]{\operatorname{\Gamma}_{#1}Q}
\newcommand{\twi}{\Psi} % spherical twist
\def\bb{\psi}
\def\ss{{\mathrm{ss}}}
\newcommand{\sslash}{\mathbin{/\mkern-6mu/}}
\title[Q\MakeLowercase{Stab on} CY-$\XX$]
{$q$-Stability conditions on Calabi-Yau-$\XX$ categories}
\author{Akishi Ikeda}
\address{AI:
Department of Mathematics,
    Josai University,
    Saitama, Japan}
\email{akishi@josai.ac.jp}

\author{Yu Qiu}
\address{Qy:
	Yau Mathematical Sciences Center and Department of Mathematical Sciences,
	Tsinghua University,
    100084 Beijing,
    China.
    \&
    Beijing Institute of Mathematical Sciences and Applications, Yanqi Lake, Beijing, China}
\email{yu.qiu@bath.edu}

\subjclass[2020]{18E30,32Q26,14F08}

\begin{document}
%=========================================================
%=========================================================

%=========================================================
\begin{abstract}
    We introduce $q$-stability conditions
    $(\sigma,s)$ on Calabi-Yau-$\mathbb{X}$ categories $\mathcal{D}_\mathbb{X}$,
    where $\sigma$ is a stability condition on $\mathcal{D}_\mathbb{X}$ and $s$ a complex number.
    We prove the corresponding deformation theorem, that
    $\operatorname{QStab}_s\mathcal{D}_\mathbb{X}$ is a complex manifold of dimension $n$ for fixed $s$,
    where $n$ is the rank of the Grothendieck group of $\mathcal{D}_\mathbb{X}$ over $\mathbb{Z}[q^{\pm 1}]$.
    When $s=N$ is an integer,
    we show that $q$-stability conditions can be identified with the stability conditions on $\mathcal{D}_N$,
    provided the orbit category $\mathcal{D}_N=\mathcal{D}_\mathbb{X}/[\mathbb{X}-N]$ is well defined.
    To attack the questions on existence and deformation along the $s$ direction,
    we introduce the inducing method.
    Sufficient and necessary conditions are given, for
    a stability condition on an $\mathbb{X}$-baric heart (that is, a usual triangulated category) of $\mathcal{D}_\mathbb{X}$
    to induce $q$-stability conditions on $\mathcal{D}_\mathbb{X}$.
    As a consequence, we show that the space $\operatorname{QStab}^\oplus\mathcal{D}_\mathbb{X}$ of
    (induced) open $q$-stability conditions is a complex manifold of dimension $n+1$.

    Our motivating examples for $\mathcal{D}_\mathbb{X}$
    are coming from (Keller's) Calabi-Yau-$\mathbb{X}$ completions of dg algebras.
    In the case of smooth projective varieties,
    the $\mathbb{C}^*$-equivariant coherent sheaves on canonical bundles provide the Calabi-Yau-$\mathbb{X}$ categories.
    Another application is that
    we show perfect derived categories can be realized as cluster-$\mathbb{X}$ categories for acyclic quivers.
    \vskip .3cm
    {\parindent =0pt
    \it Key words:}
    $q$-deformation, stability conditions, Calabi-Yau-$\XX$ categories, cluster-$\XX$ categories

\end{abstract}
\maketitle
\tableofcontents\addtocontents{toc}{\setcounter{tocdepth}{1}}

\setlength\parindent{0pt}
\setlength{\parskip}{5pt}
%=========================================================

%=========================================================
\section{Introduction}
%=========================================================
\subsection{$q$-Deformation of stability conditions}
%=========================================================
The notion of a stability condition $\sigma=(Z,\sli)$
on a triangulated category $\D$ was introduced by Bridgeland,
motivated from Douglas' work of $\Pi$-stability of D-branes in string theory.
The data consists of
\begin{itemize}
\item a central charge $Z\colon K(\D)\to\CC$, where $K(\D)$ is the Grothendieck group, and
\item a slicing $\sli=\{\sli(\phi)\}$, where $\sli(\phi)$ is an additive/abelian subcategory of $\D$,
which is an $\RR$-refinement of t-structures.
\end{itemize}
A key result established by Bridgeland \cite{B1} is that all stability conditions form a complex manifold $\Stab\D$,
where the local coordinate is given by the central charge
\[
    Z\in\Hom_{\ZZ}(K(\D),\CC).
\]
In most of the cases from algebras, $K(\D)\cong\ZZ^n$ and $\dim_\CC\Stab\D=n$.
Due to works of many people,
the theory of stability conditions has been related/applied to various subjects in mathematics,
in particular, mirror symmetry, cluster algebras, moduli spaces and Donaldson-Thomas theory
(e.g. \cite{B2,B3,BQS,BM,BS,GMN,HKK,I,KQ1,KoSo,QQ,QW,To}).

In this paper, we introduce $q$-stability conditions $(\sigma,s)$ on a class of triangulated categories $\D_\XX$,
consisting of a stability condition $\sigma$ and a complex parameter $s\in\CC$.
We consider it as a $q$-deformation of a Bridgeland stability condition.
Such a category $\D_\XX$ admits a distinguished auto-equivalence $\XX$ (another shift) satisfying
$$
    K(\D_\XX)\cong \Zq^n,\quad \Zq=\ZZ[q,q^{-1}].
$$
Here the $\Zq$-module structure on $K(\D_\XX)$ is given by the action $q[M]=[M[\XX]]$.
The compatible/extra condition on $(\sigma,s)$ is
\[
    \XX ( \sigma)=s \cdot \sigma,
\]
where the left hand side is the $\XX$ action and the right hand side is the $\CC$-action.
To spell this out,
\begin{itemize}
\item the central charge is $\Zq$-linear, i.e.
\[
    Z \in \Hom_\Zq(K(\D_{\XX}),\CC_s),
\]
where $\CC_s$ is still the complex plane but with the $\Zq$-module structure through
the action $q(z)=e^{\bi \pi s}\cdot z$.
\item the slicing is compatible with $Z$ under the action of $\XX$, i.e.
$$\sli(\phi + \Re( s)) = \sli(\phi)[\XX].$$
\end{itemize}
This type of equation was considered by Toda \cite{To} to calculate Gepner/orbit point
in $\CC\backslash\Stab\D/\Aut$.
When fixing $s$, the space $\QStab_s\D$ of $q$-stability conditions forms a complex manifold of dimension $n$ that behaves as the usual spaces of stability conditions with finite dimension. In particular, when $s=N$ is an integer and the orbit category $\D_N\colon=\D_\XX/[\XX-N]$ is well-defined, $\QStab_s\D$ can be embedded in $\Stab\D_N$ indeed.

%Our key construction of $q$-stability conditions is via $\XX$-baric heart,
%the $\XX$-analogue of usual heart.
%An important difference is that an $\XX$-baric heart is a triangulated (not abelian) category.
%We also introduce an interesting function $\gldim$ on the spaces of stability conditions
%as generalization of the usual notion of global dimension of algebras,
%which will be systematically studied in the twin paper \cite{Q3}.
%The philosophy is that \begin{itemize}
%\item a triangulated category is (generated by) $\ZZ[1]$ copy of an abelian category (its heart);
%\item an $\XX$-triangulated category is (generated by) $\ZZ[\XX]$ copy of
%a (usual) triangulated category (its $\XX$-baric heart).
%\end{itemize}

The main difficulty of this paper lies in finding the proper/correct definition
(such as $q$-stability conditions and global dimension function) so that
one can prove the corresponding results with interesting motivations/applications
(to the usual Calabi-Yau-$N$ case and to $q$-deformation of quadratic differentials).
The following two subsections are devoted to explain these motivation/application,
cf. various further developments in \cite{Q3,IQ2,FLLQ,IQZ,Q20,IOST}.

%=========================================================
\subsection{Frobenius structures}
%=========================================================
Our original motivation was to understand the link between
the space of stability conditions, for a Dynkin type $Q$,
and the Frobenius structure (which was called the flat structure in \cite{Sa1})
on the unfolding space of the corresponding singularity
constructed by Saito \cite{Sa1,SaTa}.
This link was conjectured by Takahashi, cf. comments in \cite{B3}.

Let $\hs$ be the Cartan subalgebra of the finite dimensional complex simple Lie algebra
$\mathfrak{g}$ corresponding to (an ADE quiver) $Q$ and $\hs_{\reg}$ its regular part.
Then for the `Calabi-Yau-$\infty$` category $\D_\infty(Q):=D^b(\k Q)$, we expect
\begin{gather}\label{eq:00}
    \Stab\D_\infty(Q)\cong\hs/W
\end{gather}
for the Frobenius manifold $\hs/W$, where $W$ is the Weyl group.
This has been proved for type $A$ in \cite{HKK} (cf. \cite{BQS} for $A_2$ case).
On the other hand, for the Calabi-Yau-$N$ category $\D_N(Q):=\D_{fd}(\qq{N})$,
we expect
\begin{gather}\label{eq:0}
    \Stab\D_N(Q)/\ST_N(Q)\cong\hs_{\reg}/W,
\end{gather}
where $\ST_N(Q)$ is the spherical twist group that can be identified with
the Artin/braid group $\Br_Q$ \cite{QW}.
This has been proved for type $A$  in \cite{I} (cf. \cite{BQS} for $A_2$ case).
In the correspondence \eqref{eq:0}, the central charges $Z=Z(N)$ correspond to
the twisted period maps $P_\nu$ with the parameter $\nu=(N-2)/2$ \cite[(5.11)]{Dub1}.
The relation between the twisted period maps with the parameter $\nu$ and
the central charges of Calabi-Yau-$N$ categories was first conjectured in \cite{B4}
for the canonical bundle of the projective space.
Notice that $N$ requires to be an integer $(\ge2)$ in the previous settings for $\Stab\D_N(Q)$
where the category $\D_N(Q)$ is well-defined.
However, for the almost Frobenius structure on $\hs_{\reg}/W$,
the twisted period maps can be defined for any complex parameter $\nu$.
This motivates us to produce a corresponding space of stability conditions for a complex parameter $s$
with the formula $\nu=(s-2)/2$.
Our construction of the $s$-fiber of the space $\QStab_s\DXQ$ of $q$-stability conditions
on $\DXQ:=\D_{fd}(\qq{\XX})$ provides one solution to this question.
Here $\qq{\XX}$ is Calabi-Yau-$\XX$ Ginzburg differential (double) graded algebra,
using Keller's construction \cite{Kel1}.
Moreover, the distinguished auto-equivalence $\XX$ in
the Calabi-Yau-$\XX$ categories is given by the (extra) grading shift.
%which corresponds to Adams grading in topology.

Finally, we remark that the twisted period map of the almost Frobenius structure on
$\hs_{\reg}/W$ can be identified with the period map associated with the primitive form of
the corresponding ADE singularity in the theory of Saito \cite{Sa1}.
The complex parameter $s \in \CC$ he introduced precisely corresponds to our $s$ for $q$-stability conditions.
Thus our construction is a first step of categorification of
his period map with the parameter $s$.

%=========================================================
\subsection{Mirror symmetry}
%=========================================================
The main method used to prove \eqref{eq:00} and \eqref{eq:0}
is to realize stability conditions as quadratic differentials.
Such an observation was first made by Kontsevich and Seidel years ago
and two crucial works have been carried out in \cite{BS,HKK}.
The idea is related to mirror symmetry.
More precisely, in Kontsevich's homological mirror symmetry \cite{K},
there is a conjectural derived equivalence (in general)
\[
    \D^b\Fuk(X) \cong \D^b(\coh X^{\vee}),
\]
where $X$ is a Calabi-Yau manifold with derived Fukaya category $\D^b\Fuk(X)$ on the left hand side (A-side)
and $X^{\vee}$ is mirror partner of $X$ with bounded derived category $\D^b(\coh X^{\vee})$ on the right hand side (B-side).
In the mathematical aspects of string theory,
one expects that the complex moduli space $\hua{M}_{\mathrm{cpx}}(X)$ for $X$ can be embedded into
(a quotient of) the space $\Stab \D^b(\coh{X^\vee})$ of stability conditions.
The key correspondence in this conjectural relation is
the central charge of a stable object $S^{\vee}$ in $\D^b(\coh{X^\vee})$,
is given by the integral of the complex structure via the corresponding Lagrangian $S$ in $X$:
\begin{gather}
    Z(S^{\vee})=\int_{ S } \Omega,
\end{gather}
where $\Omega$ is a canonical holomorphic volume form of the Calabi-Yau manifold $X$.
(For details of this expectation, we refer to \cite[\S~5]{ABCDGKMSSW} and \cite{B5}.)

While the general case is hard, the surface case is attackable:
the complex structure simplifies to quadratic differentials (on Riemann surfaces).
In the Calabi-Yau-$3$ case, Bridgeland-Smith \cite{BS} prove that
\begin{gather}\label{eq:BS}
    \Stap\D_3(\surf)/\Aut\cong\Quad_3(\surf),
\end{gather}
where $\surf$ is a marked surface, $\D_3(\surf)$ the associated Calabi-Yau-$3$ category
and $\Quad_3(\surf)$ the moduli space of quadratic differentials on $\surf$
(with predescribed singularities of GMN type).
Here $\D_3(\surf)$ is a subcategory of some derived Fukaya category \cite{S}.
In the non-Calabi-Yau case, Haiden-Katzarkov-Kontsevich \cite{HKK} prove that
\begin{gather}
    \Stap\D_\infty(\surf)/\Aut\cong\Quad_\infty(\surf),
\end{gather}
where $\surf$ is a flat surface, $\D_\infty(\surf)=\TFuk(\surf)$ the associated
topological Fukaya category
and $\Quad_\infty(\surf)$ the moduli space of quadratic differentials on $\surf$
(with predescribed singularities of exponential type).

One key observation is that $\D_\infty(\surf)$ should be thought as Calabi-Yau-$\infty$,
which is a philosophy developed in \cite{KQ,Q2,BQS}.
When $\surf$ is a disk, we have
\begin{gather*}
    \D_\infty(\surf)=\D_\infty(A_n),\quad
    \D_3(\surf)=\D_3(A_n).
\end{gather*}
We will focus on the surface case, relating works \cite{BS,HKK} in the sequel \cite{IQ2}
by introducing the Calabi-Yau-$\XX$ categories of quivers with superpotential.

Note that the prototype of our Calabi-Yau-$\XX$ categories is the one considered in \cite{KhS}
(cf. its mirror counterpart in \cite{ST}).
Presumedly, the construction in Section~\ref{sec:cool} (cf. \cite[S~4]{IQ2}) on the B-side
and the one in \cite{Se} for the A-side should give
a Calabi-Yau-$\XX$ version of homological mirror symmetry of \cite[Thm.A and Thm.B]{LP}.

%=========================================================
\subsection{Application to cluster theory}
%=========================================================
There is a close link between stability conditions and cluster theory, cf. \cite{KoSo}.
For instance, cluster theory (in particular cluster exchange graph)
plays a key role in the proof of \eqref{eq:BS} in \cite{BS},
as well as in proving simply connectedness/contractibility of spaces of stability conditions in \cite{Q2,QW,KQ1}.
More precisely,
the cell structure of spaces of stability conditions is encoded by Happel-Reiten-Smal{\o} tilting,
which corresponds to mutation of silting objects in the perfect derived categories (cf. \cite{K3}).
On the other hand, in the categorification of cluster algebras,
the Calabi-Yau-2 cluster categories $\C_2(Q)$
are introduced by Buan-Marsh-Reineke-Reiten-Todorov \cite{BMRRT,K2}.
It is generalized by Amiot-Guo-Keller \cite{A,Kel1,Guo} as Verdier quotient
\begin{gather}\label{eq:C1}
    \C_{N-1}(Q)=\per\qq{N}/\D_{fd}(\qq{N}), \quad N\geq2(\in\ZZ)
\end{gather}
of the Calabi-Yau-$N$ Ginzburg dga $\qq{N}$.
The mutation of cluster algebras corresponds to the mutation of cluster tilting objects
in the cluster categories,
which is also closely related to mutation of silting objects (cf. e.g. \cite{KQ}).
One expects the perfect derived categories $\DQ\cong\per\k Q$ (of an acyclic quiver $Q$)
would be the Calabi-Yau-$\infty$ cluster categories $\C_{\infty}(Q)$ as
\begin{gather}\label{eq:C2}
    \DQ `\approx~' \lim_{m\to\infty}\C_m(Q),
\end{gather}
in the sense that the fundamental domains in $\DQ$ for the orbit quotient $\DQ\to\C_m(Q)$
can be chosen with inclusion relation and the limit of which is $\DQ$.
The corresponding statement for the spaces of stability conditions is (\cite[Thm.~6.2]{Q2})
\[\Stab\DQ \cong \lim_{N\to\infty} \Stab\D_N(Q)/\Br_Q.\]

As an application of our Calabi-Yau-$\XX$ construction, we show that
in fact, $\DQ$ can be realized as cluster-$\XX$ category, i.e.
as Verdier quotient $\per\qq{\XX}/\D(\qq{\XX})$ of $\qq{\XX}$, by relating
\eqref{eq:C1} and \eqref{eq:C2} by completing a commutative diagram \eqref{eq:C12} below
of short exact sequences of triangulated categories.
This unifies the theory of cluster tilting and silting,
which provides a new perspective to study these categories.

%=========================================================
\subsection{Content}
%=========================================================
In Section~\ref{sec:X},
we construct Calabi-Yau-$\XX$ categories from quivers and from coherent sheaves as our motivating examples.
In Section~\ref{sec:QS} to Section~\ref{sec:ind}, we introduce $q$-stability conditions (Definition~\ref{def:Xstab})
and we prove the following theorem
(cf. Theorem~\ref{thm:localiso2}, Theorem~\ref{thm:reduction} and Theorem~\ref{thm:manifold}):
\begin{thma}
Let $\D_\XX$ be a category satisfies Assumption~\ref{assumption:R},
$n$ the rank of $K\D_\XX$ over $R=\ZZ[q^{\pm1}]$ and $s$ a complex number.
We have the following.
\begin{itemize}
  \item The space $\QStab_s\D_\XX$ of $q$-stability conditions is a complex manifold of dimension $n$
  with local coordinate
  \begin{gather*}
    \begin{array}{ccc}
    \mathcal{Z}_s \colon \QStab_s\D_{\XX} &\longrightarrow& \Hom_\Zq(K(\D_{\XX}),\CC_s),\\
    \quad((Z,\hh{P}),s)& \mapsto& Z.
    \end{array}
  \end{gather*}
  \item If the orbit category $\D_N\colon=\D_{\XX} \sslash [\XX-N]$ is well-defined (i.e. $N$-reductive),
  then there is a canonical injection of complex manifolds
    $$\iota_N \colon \QStab_N(\D_{\XX}) \to \Stab\D_N,$$ whose image is open and closed.
  \item If $\D_\XX$ is Calabi-Yau-$\XX$, then
  the space $\OStab\D_\XX$ of induced open $q$-stability conditions is a complex manifold of dimension $n+1$.
\end{itemize}
\end{thma}
In Section~\ref{sec:cluster}, we show that the perfect derived category of an acyclic quiver
can be realized as a cluster-$\XX$ category that fits into the following story (Corollary~\ref{cor:SES}).
\begin{thmb}
Let $Q$ be an acyclic quiver and $N\ge2$ an integer.
Then there is a commutative diagram:
\begin{equation}\label{eq:C12}
\xymatrix{
    0 \ar[r] &
        \DXQ \ar[r]\ar[d]^{_{\sslash[\XX-N]}} &
        \per\qq{\XX} \ar[r]\ar[d]^{_{\sslash[\XX-N]}} &
        \per\k Q \ar[r]\ar[d]^{_{/\tau[2-N]}} &0 \\
    0 \ar[r] & \DNQ \ar[r]& \per\qq{N} \ar[r] & \C_{N-1}(Q) \ar[r] &0
}.\end{equation}
\end{thmb}
In Appendix~\ref{sec:cat}, we discuss the categorification of $q$-deformed root lattices.
%=========================================================
\subsection*{Acknowledgments}
%=========================================================
We would like to thank Tom Bridgeland, Alastair King, Bernhard Keller, Tatsuki Kuwagaki, Kyoji Saito,
Yukinobu Toda, Dong Yang and Yu Zhou for inspirational discussions and advices.
This work is is supported by
National Key R\&D Program of China (No. 2020YFA0713000),
Hong Kong RGC 14300817 (from Chinese University of Hong Kong),
Beijing Natural Science Foundation (Z180003)
World Premier International Research Center Initiative (WPI initiative), MEXT, Japan and
JSPS KAKENHI Grant Number JP16K17588.

%=========================================================
\section{Calabi-Yau-$\XX$ categories}\label{sec:X}
%=========================================================
All the modules and categories will be over $\k$, an algebraically closed field.
%=========================================================
\subsection{Differential double graded algebras}
%=========================================================
The aim of this section is to introduce differential double graded algebras
which are differential graded algebras graded by $\ZZ \times \ZZ$.
For convenience, we identify $\ZZ \times \ZZ$ with
the rank two free module $\ZZ \oplus \ZZ \XX$
spanned by the basis $1$ and $\XX$.
Thus a pair of integers
$(m,l) \in \ZZ \times \ZZ$ is identified with $m+l \XX \in \ZZ \oplus \ZZ \XX$.

A {\it differential double graded (ddg) algebra} $A$ is a graded algebra
\[
A :=\bigoplus_{m,l \in \ZZ}A^{m+l\XX}
\]
graded by $\ZZ \oplus \ZZ \XX$
with the differential $\diff \colon A^{m+l\XX} \to A^{m+1+l\XX}$ of degree $1$.
We can also define ddg-modules over $A$
similar to usual dg-modules.
For ddg-module $M=\oplus_{m,l}M^{m+l \XX}$ with the
differential $\diff$, the degree shift
$M^{\prime}=M[k+j\XX]$ is defined by
\[
    (M^\prime)^{m+l\XX}:=M^{m+k+(l+j)\XX}
\]
with the new differential $\diff^{\prime}:=(-1)^k \diff$.

Let $\D(A)$ be the derived category of ddg-modules over $A$.
The {\it perfect derived category} $\Per A \subset \D(A)$
is the smallest full triangulated subcategory containing
$A$ and closed under taking direct summands.

%=========================================================
\subsection{Calabi-Yau algebras and Calabi-Yau categories}
%=========================================================
In this section, we recall the definition of Calabi-Yau algebras.
As in the previous section, let $A$ be a ddg-algebra.
The {\it enveloping algebra of $A$} is defined by $A^e:=A \otimes A^{op}$.
The ddg-algebra $A$ has the $A^e$-module structure given by the two-sided action of $A$.

A ddg-algebra $A$ is called {\it homologically smooth} if
\[
A \in \Per A^e.
\]

Let $A$ be a homologically smooth ddg-algebra.
The {\it inverse dualizing complex} $\Theta_A$ is defined by the cofibrant replacement of
\[
    \RHom_{A^e}(A,A^e)
\]
considered as an object of $\D(A^e)$,

We treat $\Theta_A$ as an object in $\Per A^e$.
Denote by $\D_{fd}(A)$ the full subcategory of $\D(A)$ consisting of ddg-modules
with finite dimensional total cohomology, i.e.
\[
\D_{fd}(A):=\{\,M \in \D(A)\,\vert\, \sum_k \dim H^k(M)<\infty \,\}.
\]
By using $\Theta_A$, the Serre duality is described as follows.
\begin{lemma}[\cite{Kel1}, Lemma 3.4]
\label{lem:Serre}
Let $A$ be a homologically smooth ddg-algebra and $\Theta_A$ be its
inverse dualizing complex.
Then, for  any ddg-modules $M,N \in \D_{fd}(A)$,
there is a canonical isomorphism
\[
\Hom_{\D(A)}(M \otimes_A \Theta_A, N) \iso \mathrm{D}\Hom_{\D(A)}(N,M).
\]
\end{lemma}

The following Calabi-Yau property of a ddg-algebra is due to
Ginzburg and Kontsevich \cite[Def.~3.2.3]{G}.
\begin{definition}
\label{defi:CY}
Let $\mathcal{N} \in \ZZ\oplus \ZZ\XX$.
A ddg-algebra $A$ is called
a {\it Calabi-Yau-$\mathcal{N}$}  algebra if $A$ is homologically smooth and
\[
\Theta_A \iso A[-\mathcal{N}]
\]
in $\Per A$.

A triangulated category $\D$ is called \emph{Calabi-Yau-$\hh{N}$} (CY-$\hh{N}$)
if, for any objects $X,Y$ in $\hh{D}$ we have a natural isomorphism
\begin{gather}\label{eq:serre}
    \mathfrak{S}:\Hom (X,Y)
        \xrightarrow{\sim} \mathrm{D}\Hom (Y,X[\hh{N}]).
\end{gather}
\end{definition}

If $A$ is a Calabi-Yau-$\mathcal{N}$ algebra, then $\D_{fd}(A)$
becomes a Calabi-Yau-$\mathcal{N}$ category.

%=========================================================
\subsection{Calabi-Yau-$\XX$ completions for ddg-algebras, following Keller}\label{sec:CYX}
%=========================================================
As in the previous section, we consider ddg-algebras indexed by $\ZZ \oplus \ZZ\XX$.

\begin{definition}
Let $A$ be a homologically smooth ddg-algebra
and $\Theta_A$ be its inverse dualizing complex.
The
{\it  Calabi-Yau-$\XX$ completion}
of $A$ is defined by
\[
\Pi_{\XX}(A):=\mathrm{T} _A(\theta)=
A \oplus \theta \oplus (\theta \otimes_A \theta) \oplus \cdots
\]
where $\theta:=\Theta_A[\XX-1]$ and
$\mathrm{T} _A(\theta)$ is the tensor algebra of $\theta$
over $A$.
\end{definition}

Keller's result can be adopted in this setting.
\begin{theorem}[\cite{Kel1}, Theorem 6.3 and \cite{K4}, Theorem 1.1]
%\cite[Theorem.~6.3]{Kel1}
For a homologically smooth ddg-algebra $A$,
the Calabi-Yau-$\XX$ completion $\Pi_{\XX}(A)$
becomes a Calabi-Yau-$\XX$ algebra.
In particular, $\D_{fd}(\Pi_{\XX}(A))$ is Calabi-Yau-$\XX$.
\end{theorem}

We also have the following.

\begin{lemma}\cite[Lemma.~4.4]{Kel1}\label{lem:cy}
The canonical projection (on the first component) $\Pi_{\XX}(A)\to  A$ induces an
Lagrangian-$\XX$ immersion
\[
    \hh{L}\colon \D_{fd}(A)\to\D_{fd}( \Pi_{\XX}(A) ),
\]
in the sense that for $L,M\in\D_{fd}(\ha)$, we have
\begin{gather}\label{eq:HOM}
    \RHom_\Pi(\hh{L}(L),\hh{L}(M))=\RHom_{\ha}(L,M)
        \oplus D\RHom_{\ha}(M,L)[-\XX].
\end{gather}
\end{lemma}

\begin{remark}
All statements of this section work not only for $\XX$ but also
for any $\mathcal{N} \in \ZZ \oplus \ZZ\XX$. However in this paper,
we mainly deal with the case $\mathcal{N}=\XX$.
\end{remark}

%=========================================================
\subsection{Acyclic quiver case}
%=========================================================
In the case when $A$ is a path algebra of an acyclic quiver, the
Calabi-Yau-$\XX$ completion $\Pi_{\XX}(A)$ has the following explicit description,
known as the Ginzburg Calabi-Yau algebra \cite{G}.

\begin{definition}\cite{G,Kel1}\label{def:CYXQ}
Let $Q=(Q_0,Q_1)$ be a finite acyclic quiver with
vertices $Q_0=\{1,\dots,n\}$ and arrows $Q_1$.
We introduce the {\it Ginzburg Calabi-Yau-$\XX$ ddg algebra}
$\Gamma_{\XX}Q :=(\k \overline{Q},d)$ as follows.
Define a $\ZZ\oplus\ZZ\XX$-graded quiver $\overline{Q}$ with vertices
$\overline{Q}_0=\{1,\dots,n\}$ and following arrows:
\begin{itemize}
\item an original arrow $a \colon i \to j \in Q_1$ (degree $0$);
\item an opposite arrow $a^* \colon j \to i$
for the original arrow $a \colon i \to j \in Q_1$ (degree $2-\XX$);
\item a loop $t_i=e_i^*$ for each vertex $i \in Q_0$ (degree $1-\XX$).
\end{itemize}
Let $\k \overline{Q}$ be a $\ZZ \oplus \ZZ\XX$-graded
path algebra of $\overline{Q}$,
and define a differential $\diff \colon \k \overline{Q} \to \k \overline{Q}$ of degree $1$ by
\begin{itemize}
\item $\diff a =\diff a^*= 0$ for $a \in Q_1$;
\item $\diff t_i = e_i \,\left(\sum_{a \in Q_1}(aa^* -a^*a)\right)\,e_i$;
\end{itemize}
where $e_i$ is the idempotent at $i \in Q_0$. Thus
we have the ddg algebra $\Gamma_{\XX}Q =(\k \overline{Q},\diff)$.
Note that the $0$-th homology is given by
$H^0(\Gamma_{\XX}Q) \cong \k Q$.
\end{definition}

Denote by $\DQ\colon=\D^b(\k Q)$ the bounded derived category of $\k Q$ and
$$\DXQ\colon=\D_{fd}(\Gamma_{\XX}Q)$$ the finite-dimensional derived category of $\Gamma_{\XX}Q$.
Then we have the following Calabi-Yau-$\XX$ version of Keller's Calabi-Yau completion.

\begin{corollary}\cite[Thm.~6.3]{Kel1}\label{cor:Lag}
The Calabi-Yau-$\XX$ completion $\Pi_{\XX}(k Q)$ of the path algebra $k Q$ is
isomorphic to the Ginzburg Calabi-Yau-$\XX$ algebra $\Gamma_{\XX}Q$.
\end{corollary}

In particular, $\D_\XX(Q)$ is Calabi-Yau-$\XX$.
There is an Lagrangian immersion
\begin{gather}\label{eq:L_Q}
    \hh{L}_Q\colon\DQ\to\DXQ.
\end{gather}
By abuse of notation, we may not distinguish $\DQ$ and $\hh{L}_Q(\DQ)$.

%=========================================================
\subsection{$\CC^*$-equivariant coherent sheaves on canonical bundles}\label{sec:cool}
%=========================================================

Let $X$ be a smooth projective variety over $\CC$ and
$\D^b (\Coh X)$ the bounded derived category of  coherent sheaves
on $X$. Due to \cite{BvdB}, there is a classical generator
$G \in \D^b (\Coh X)$ and \cite{O} gives the direct description
$G=\oplus_{i=1}^{d+1}L^i$ where $L$ is the ample line bundle on $X$
and $d=\dim X$.
Consider the endomorphism dg-algebra $A:=\RHom(G,G)$.
Then we have the equivalence of derived categories
\begin{gather}
    F:=\RHom(G,-) \colon \D(\mathrm{Qcoh} X)\to \D(A).
\end{gather}
Since $G$ is a classical generator of $\D^b (\Coh X)$, the restriction of $F$ gives the equivalence
\[
    F \colon \D^b (\Coh X) \xrightarrow{\cong} \Per A.
\]
By \cite[Lemma 3.4.1]{BvdB}, the object
$G \boxtimes G^*$ is a generator of $ \D^b (\Coh (X \times X))$. Thus we also have the equivalence
\begin{gather}
    F^{e}:=\RHom(G \boxtimes G^*,-) \colon
     \D^b (\Coh (X \times X)) \xrightarrow{\cong} \Per A^e
\end{gather}
since $\RHom(G \boxtimes G^*,G \boxtimes G^*)\cong A \otimes A^{op}=A^e $.

We note that $\Per X \times X \cong D^b(\Coh (X \times X))$ since $X$ is smooth.
For a perfect object $\E \in \Per X \times X$, define the derived dual by
\[
    \E^{\vee}:=\hom_{\OO_{X \times X}}(\E,\OO_{X \times X})
     \in \Per X \times X.
\]
Similarly for a perfect object $M \in \per A^e$,
let $M^{\vee}:=\Hom_{A^e}(M,A^e)$ be its derived dual.
By definition, $M^{\vee}$ has a natural left
$A^e$-module structure (so the right $(A^e)^{op}$-module structure).
We note $(A^e)^{op}\cong A^{op}\otimes A$.
Through the algebra homomorphism
\[\begin{array}{ccc}
    \tau \colon A\otimes A^{op}&\to& A^{op}\otimes A,\\
    \quad a \times b &\mapsto& (-1)^{|a||b|}b \otimes a
\end{array}\]
exchanging two components, we can define the right $A^e$-module structure on $M^{\vee}$.
Thus $M^{\vee}$ can be regarded as an object in $\per A^e$.

\begin{lemma}
\label{lem:canonical}
Let $\Delta \colon X \to X \times X$ be the diagonal embedding. Then
\[
F^{e}(\Delta_*\mathcal{K}_X^{-1})=\Theta_{A}[d].
\]
\end{lemma}
\begin{proof}
First we check that $F^e({\Delta_* \mathcal{O}_X})=A$. By definition,
$G \boxtimes G^*=p_1^*G \otimes p_2^* G^*$ where $p_i$ are
projections from $X \times X $ to the $i$-th component. Then
\begin{align*}
F^e(\Delta_* \mathcal{O}_X)
&=\RHom(p_1^*G \otimes p_2^* G^*, \Delta_* \mathcal{O}_X) \\
&=\RHom(p_1^*G ,p_2^* G \otimes {\Delta_* \mathcal{O}_X}) \\
&=\RHom(p_1^*G ,\Delta_*( G \otimes \mathcal{O}_X)) \\
&=\RHom(\Delta^* p_1^*G , G )=A.
\end{align*}
By definition, the derived dual $A^{\vee}$ in $\per(A^e)$ is $\Theta_A$. On the other hand,
the derived dual $(\Delta_* \mathcal{O}_X)^{\vee}$ in $\per X\times X$
can be computed as follows.
By using the Grothendieck-Verdier duality, we have
\[
(\Delta_*\OO_X)^{\vee}=\hom(\Delta_* \OO_X, \OO_{X \times X})
\cong \Delta_* \hom (\OO_X, \Delta^* \OO_{X\times X} \otimes \K_{\Delta}[-d])
\]
where $d=\dim X$ is the relative dimension of $\Delta$, and
\[
    \K_{\Delta}:=\K_X \otimes \Delta^* \K_{X \times X}^{-1}
\]
is the relative dualizing bundle of $\Delta$.
Note that $\Delta^* \OO_{X \times X} \cong \OO_X$. Thus we have
\begin{align*}
\Delta_* \hom_{\OO_X} (\OO_X, \Delta^* \OO_{X \times X} \otimes \K_{\Delta}[-d])\cong
\Delta_*(\K_X \otimes \Delta^* \K_{X \times X}^{-1}[-d]).
\end{align*}
Finally, the adjunction formula for $\Delta \colon X \to X \times X$ implies
\[
    \K_X \cong \Delta^* \K_{X \times X} \otimes \bigwedge^{\rank \mathcal{N}}\mathcal{N}
\]
where $\mathcal{N}$ is the normal bundle and in this case, $\mathcal{N} \cong \mathcal{T}_X$.
So $\Delta^* \K_{X \times X}\cong \K_X^{\otimes 2}$ and hence
we have $(\Delta_* \mathcal{O}_X)^{\vee} \cong \K_X^{-1}[-d]$.
Since $F^e$ commutes with taking the derived dual by next lemma,
the result follows.
\end{proof}
We  show $F^e$ commutes with taking the derived dual.
\begin{lemma}
\label{lem:derived_dual}
There is an isomorphism of right $A^e$-modules
\[
F^e(\E^{\vee})\cong F^e(\E)^{\vee}
\]
for $\E \in \per (X\times X) $.
\end{lemma}
\begin{proof}
The left hand side is
\[
F^e(\E^{\vee})=\Hom(G \boxtimes G^{\vee}, \E^{\vee}).
\]
The right hand side is
\begin{align*}
F^e(\E)^{\vee}&=\Hom_{A^e}(\Hom_{X \times X}(G \boxtimes G^{\vee},\E),A^e)\\
&\cong \Hom_{A^e}(\Hom_{X \times X}(G \boxtimes G^{\vee},\E),
\Hom_{X \times X}(G \boxtimes G^{\vee},G \boxtimes G^{\vee})) \\
&\cong \Hom_{X \times X}(\E, G \boxtimes G^{\vee}) \\
&\cong  \Hom_{X \times X}( (G \boxtimes G^{\vee})^{\vee},\E^{\vee})\\
&\cong  \Hom_{X \times X}( G^{\vee} \boxtimes G,\E^{\vee}).
\end{align*}
Clearly, $F^e(\E^{\vee})$ is isomorphic to $F^e(\E)^{\vee}$ as complexes of vector spaces.
Since $A \otimes A^{op}\cong \Hom_X(G,G)\otimes \Hom_X(G^{\vee},G^{\vee})$ acts on this
through the exchange
\[
\tau \colon  \Hom_X(G,G)\otimes \Hom_X(G^{\vee},G^{\vee})\to
 \Hom_X(G^{\vee},G^{\vee}) \otimes  \Hom_X(G,G),
\]
$A^e$-module structures coincide.
\end{proof}

\begin{remark}
Here we show Lemma \ref{lem:canonical} for a smooth projective variety.
However, the result also holds for a smooth quasi-projective variety.
See \cite[Proposition 3.10]{HK}.
\end{remark}

Let \[
p_{ij} \colon X \times X \times X \to X \times X
\]
be projections on the $i$-the and $j$-th component for $1 \le i<j \le 3$.

Define the convolution product
\[
* \colon \D^b(\Coh (X\times X))\times \D^b(\Coh (X\times X))  \to \D^b(\Coh (X\times X))
\]
by
\[
(\E,\mathcal{F})\mapsto \E*\mathcal{F}:=
p_{13 *}(p_{12}^* \E \otimes p^*_{23}\mathcal{F}).
\]
Then $\D^b(\Coh (X\times X))$ has the monoidal structure by the convolution product
with the unit object $\Delta_* \OO_X$.
On the other hand, the category $\per A^e$ has the monoidal structure
by the tensor product $M \otimes_A N$ for $M,N \in \per A^e$ with the unit object $A$.
It is easy to check the following.
\begin{lemma}
\label{lem:monoidal}
\indent
\begin{itemize}
\item The functor
$\Delta_* \colon \D^b(\Coh X) \to \D^b(\Coh (X\times X))$ is a monoidal functor.
\item The functor $F^{e}\colon
 \D^b (\Coh (X \times X)) \to \Per A^e$ is a monoidal functor.
\end{itemize}
\end{lemma}

Let $Y:=\mathrm{V}(\mathcal{K}_X)$ be the total space of the canonical bundle
of $X$. We note that as a scheme, $Y$ is given by
\[
Y=\mathrm{Spec}_X \mathrm{T} \mathcal{K}_X^{-1}
\]
where $\mathrm{T} \mathcal{K}_X^{-1}$ is the tensor algebra of the inverse of the canonical sheaf.
(Since $\mathcal{K}_X^{-1}$ is rank one, the tensor algebra
$\mathrm{T} \mathcal{K}_X^{-1}$ coincides with the symmetric algebra
$\mathrm{S} \mathcal{K}_X^{-1}$.)
We note that the bounded derived category of coherent sheaves on $Y$
can be identified with the bounded derived category of finite rank
$\mathrm{T} \mathcal{K}_X^{-1}$-modules on $X$:
\[
\D^b (\Coh Y) \cong \D^b(\mathrm{mod-} \mathrm{T} \mathcal{K}_X^{-1} ).
\]
Since Lemma \ref{lem:canonical} and Lemma \ref{lem:monoidal} imply
\[
F^e(\Delta_* \mathrm{T} \mathcal{K}_X^{-1})=\mathrm{T}(\Theta_A[d])=\Pi_{d+1}(A),
\]
we have an equivalence
\[
\D^b(\Coh Y) \cong \Per \Pi_{d+1}(A).
\]
Denote by $\D^b_c(Y)$ the derived category of coherent sheaves on $Y$ with compact
support cohomology. Then we also have an equivalence
\[
\D^b_c(\Coh Y) \cong \D_{fd}(\Pi_{d+1}(A)).
\]

Next we consider the fiber scaling action of
$\CC^*$ on $Y$ and $\CC^*$-equivariant coherent sheaves on $Y$.
For a $\CC^*$-module $M$, we have the weight decomposition
\[
M:=\bigoplus_{m \in \ZZ}M^m
\]
where $t \in \CC^*$ acts on $M^d$ by $t^d$. We denote by $M\{1\}$ the weight one shift of $M$,
namely $M\{1\}^m:=M^{m+1}$.
The $\CC^*$-action on $Y$ induces the weight decomposition on $\mathrm{T} \mathcal{K}_X^{-1}$:
\[
(\mathrm{T} \mathcal{K}_X^{-1})_{\mathrm{gr}}:=\mathrm{T}(\mathcal{K}_X^{-1}\{1\})=
\bigoplus_{m \ge 0}\mathcal{K}_X^{-m}\{m\}
\]
where $t\in \CC$ acts on $\mathcal{K}_X^{-m}\{m\}$ as $t^{-m}$.
We regard $(\TK)_{\mathrm{gr}}$ as a $\ZZ$-graded algebra through the
above weight decomposition.
Then a $\CC^*$-equivariant coherent sheaf on $Y$ can be identified with
a finite rank $\ZZ$-graded $(\TK)_{\mathrm{gr}}$-module over $X$.
Again there is an equivalence
\[
    \D^b_{\CC^*}(Y) \cong \D^b(\mathrm{grmod-}(\TK)_{\mathrm{gr}} )
\]
where the left hand side is the bounded derived category of $\CC^*$-equivariant
coherent sheaves on $Y$ and the right hand side is
the bounded derived category of finite rank
$\ZZ$-graded $(\TK)_{\mathrm{gr}}$-modules over $X$.

Through the identification of $\mathcal{K}_X^{-1}$ with $\Theta_A[d]$,
one can define the extra $\ZZ$-grading structure on $\mathrm{T}(\Theta_A[d])$ by
\[
\mathrm{T}(\Theta_A[d])_{\mathrm{gr}}:=\bigoplus_{m \ge 0}(\Theta_A[d]\{1\})^{\otimes m}.
\]
If we set $[\XX]:=[d+1]\{1\}$, then by definition we regard
$\mathrm{T}(\Theta_A[d])_{\mathrm{gr}}$ as the Calabi-Yau-$\XX$ completion $\Pi_{\XX}(A)$.
Thus we obtain the following equivalences.
\begin{proposition}
There is an equivalence of derived categories
\begin{gather}
    \D^b_{\CC^*}(Y) \cong \Per \Pi_{\XX}(A)
\end{gather}
and the shift $[d+1]\{1\}$ in the left hand side coincides with the shift $[\XX]$ in the right hand side.
By restricting objects in the left hand side to those with compact support cohomology,
we have an equivalence
\begin{gather}
    \D^b_{c,\CC^*}(Y) \cong \D_{fd}(\Pi_{\XX}(A) ).
\end{gather}
In particular, $\D^b_{c,\CC^*}(Y)$ is Calabi-Yau-$[d+1]\{1\}$.
\end{proposition}

%\Note{
\begin{example}
Let $X$ be the projective line $\kong{P}^1$ and $Q$ be the Kronecker quiver (i.e. type $\widetilde{A_{1,1}}$).
Then there is a triangulated equivalence
\begin{gather}\label{eq:KQ}
    \D^b(\coh X) \cong \DQ.
\end{gather}
Then the proposition above implies the Calabi-Yau-$\XX$ version of this equivalence
\begin{gather}\label{eq:KQ2}
        \D^b_{c,\CC^*}(Y) \cong \DXQ
\end{gather}
where $Y=T^* \kong{P}^1$ and $\XX=[2]\{1\}$.
In this case we can describe the above category as the derived category
of the graded preprojective algebra of $Q$ as follows.
Consider the graded double quiver $\widetilde{Q}_{\mathrm{gr}}$
\[
\begin{tikzpicture}[xscale=0.5,
  arrow/.style={->,>=stealth},
  equalto/.style={double,double distance=2pt},
  mapto/.style={|->}]
\draw[font=\scriptsize](0,-1)node[above]{$a_1,a_2$};
\draw[font=\scriptsize](0,-1.3)node[below]{$a^*_1,a^*_2$};
  \node at (2.4,-1.15){$2$};
  \node at (-2.4,-1.15) {$1$};
\foreach \n/\m in {1/3}
    {\draw[arrow] (-2,-1) to (2,-1);
     \draw[arrow] (-2,-1.1) to (2,-1.1);
     \draw[arrow] (2,-1.2) to (-2,-1.2);
     \draw[arrow] (2,-1.3) to (-2,-1.3);
     }
\end{tikzpicture}
\]
with the grading $\deg a_i=0$ and $\deg a^*_i=-1$ for $i=1,2$.
Let
\[
\Pi(Q)_{\mathrm{gr}}:=\CC \widetilde{Q}_{\mathrm{gr}}
\slash (a_1 a_1^*-a_1^* a_1,a_2 a_2^*-a_2^* a_2)
\]
be the graded preprojective algebra of $Q$ and
$\D^b(\mathrm{grmod-}\Pi(Q)_{\mathrm{gr}})$ be the bounded derived category of
finite dimensional graded modules over $\Pi(Q)_{\mathrm{gr}}$.
Then this grading fits into the weight of $\CC^*$-action on $Y$ and we have
an equivalence
\[
 \D^b_{c,\CC^*}(Y) \cong \D^b(\mathrm{grmod-}\Pi(Q)_{\mathrm{gr}}).
\]
\end{example}
\begin{example}
When $X$ is the projective plane $\kong{P}^2$,
then there is a corresponding version of \eqref{eq:KQ},
where $Q$ is the following quiver
\[
\begin{tikzpicture}[xscale=0.5,
  arrow/.style={<-,>=stealth},
  equalto/.style={double,double distance=2pt},
  mapto/.style={|->}]
\node (x2) at (6,-1){};
\node (x1) at (2,-1){};
\node (x3) at (-2,-1){};
\draw[font=\scriptsize](0,-1)node[above]{$x_1,y_1,z_1$}
    (4,-1)node[above]{$x_2,y_2,z_2$};
  \node at (x1){$2$};
  \node at (x2){$3$};
  \node at (x3){$1$};
\foreach \n/\m in {1/3,2/1}
    {\draw[arrow] (x\n.150) to (x\m.30);
     \draw[arrow] (x\n.-150) to (x\m.-30);
     \draw[arrow] (x\n) to (x\m);}
\end{tikzpicture}
\]
with some commutative relations
$$a_1 b_2=b_1 a_2,\quad a,b\in\{x,y,z\}.$$
We can also upgraded that to a Calabi-Yau-$\XX$ version \eqref{eq:KQ2}.
Again we can describe the above category as follows.
Consider the following graded quiver
with graded potential $(\widetilde{Q}_{\mathrm{gr}},W_{\mathrm{gr}})$
\[
\begin{tikzpicture}[scale=0.6,
  arrow/.style={->,>=stealth},
  equalto/.style={double,double distance=2pt},
  mapto/.style={|->}]
\node (x2) at (0,2.){};
\node (x1) at (2,-1){};
\node (x3) at (-2,-1){};
  \node at (x1){$_3$};
  \node at (x2){$_2$};
  \node at (x3){$_1$};
\foreach \n/\m in {1/3}
    {\draw[arrow] (x\n.150) to (x\m.30);
     \draw[arrow] (x\n.-150) to (x\m.-30);
     \draw[arrow] (x\n) to (x\m);}
\foreach \n/\m in {2/1}
    {\draw[arrow] (x\n.150+120) to (x\m.30+120);
     \draw[arrow] (x\n.-150+120) to (x\m.-30+120);
     \draw[arrow] (x\n) to (x\m);}
\foreach \n/\m in {3/2}
    {\draw[arrow] (x\n.150-120) to (x\m.30-120);
     \draw[arrow] (x\n.-150-120) to (x\m.-30-120);
     \draw[arrow] (x\n) to (x\m);}
\draw[font=\scriptsize](-.8,1.8)node[left,rotate=60]{$x_1,y_1,z_1$}
    (.6,1.8)node[right,rotate=-54]{$x_2,y_2,z_2$}
    (0,-1.2)node[below]{$x_3,y_3,z_3$};
\draw    (-5,1)node{$\widetilde{Q}_{\mathrm{gr}}:$}
    (-1,-3)node{$\displaystyle{W=\sum_{i=1}^3 (x_iy_iz_i-x_iz_iy_i)}$};
\end{tikzpicture}
\]
where $\deg x_3, y_3, z_3=-1$ and gradings of other arrows are zero.
Then we have an associated graded Jacobi algebra
$\mathrm{J}(\widetilde{Q}_{\mathrm{gr}},W_{\mathrm{gr}})$.
Thus there is an equivalence
\[
 \D^b_{c,\CC^*}(Y) \cong
 \D^b(\mathrm{grmod-}\mathrm{J}(\widetilde{Q}_{\mathrm{gr}},W_{\mathrm{gr}})).
\]
Finally we note that its $\XX=3$-reduction implies to forget the grading
from $(\widetilde{Q}_{\mathrm{gr}},W_{\mathrm{gr}})$
since $\XX=[3]\{1\}$.
Therefore the $\XX=3$-reduction gives the equivalence
\begin{gather}\label{eq:KQ3}
        \D^b_{\kong{P}^2}(Y) \cong \D^b(\mathrm{mod-} \mathrm{J}(\widetilde{Q},W)),
\end{gather}
between the derived category of coherent sheaves on the
local $\kong{P}^2$, considered in \cite{BM},
and the derived category of finite dimensional modules over the Jacobi algebra
of $(\widetilde{Q},W)$ which is obtained by forgetting the grading of
$(\widetilde{Q}_{\mathrm{gr}},W_{\mathrm{gr}})$.
We will explore this direction/application in the future studies.
\end{example}
%\begin{example}
%In \cite{LP}, they prove the following derived equivalence
%\begin{gather}\label{eq:LP}
%    \D^b(X_\surf_\infty)\cong\TFuk(\surf_\infty),
%\end{gather}
%where $\surf_\infty$ is some flat surface (cf. \cite{HKK}) and $X_\surf_\infty$
%is the corresponding nodal stacky curve.
%In the sequel \cite{IQ2}, we construct the quiver with superpotential
%$(Q_\surf_\infty,W_\surf_\infty)$ from the flat surface
%$\surf_\infty$ and obtain the Calabi-Yau-$\XX$ version of $\TFuk(\surf_\infty)$.
%Then by applying the proposition above we should also have a
%upgraded Calabi-Yau-$\XX$ triangle equivalence
%\begin{gather}\label{eq:LPX}
%    \D^b_{c,\CC^*}(Y_\surf_\infty) \cong \D_\XX(Q_\surf_\infty,W_\surf_\infty).
%\end{gather}
%\end{example}
%}

%=========================================================
\section{$q$-Stability conditions on $\XX$-categories}\label{sec:QS}
%=========================================================
\subsection{Bridgeland stability conditions}\label{sec:BSC}
%=========================================================
First we recall the definition of Bridgeland stability conditions
on triangulated categories from \cite{B1}.
Throughout this section,
we assume that for a triangulated category $\D$, its Grothendieck group
$K(\D)$ is free of finite rank, i.e. $K(\D) \cong \ZZ^{\oplus n}$ for some $n$.

\begin{definition}
\label{def:stab}
Let $\D$ be a triangulated category.
A {\it stability condition} $\sigma = (Z, \sli)$ on $\D$ consists of
a group homomorphism $Z \colon K(\D) \to \CC$ called the {\it central charge} and
a family of full additive subcategories $\sli (\phi) \subset \D$ for $\phi \in \R$
called the {\it slicing}
satisfying the following conditions:
\begin{itemize}
\item[(a)]
if  $0 \neq E \in \sli(\phi)$,
then $Z(E) = m(E) \exp(\bi \pi \phi)$ for some $m(E) \in \R_{>0}$,
\item[(b)]
for all $\phi \in \R$, $\sli(\phi + 1) = \sli(\phi)[1]$,
\item[(c)]if $\phi_1 > \phi_2$ and $A_i \in \sli(\phi_i)\,(i =1,2)$,
then $\Hom_{\D}(A_1,A_2) = 0$,
\item[(d)]for $0 \neq E \in \D$, there is a finite sequence of real numbers
\begin{equation}\label{eq:>}
\phi_1 > \phi_2 > \cdots > \phi_m
\end{equation}
and a collection of exact triangles (\emph{Harder-Narasimhan filtration})
\begin{equation}\label{eq:HN}
0 =
\xymatrix @C=5mm{
 E_0 \ar[rr]   &&  E_1 \ar[dl] \ar[rr] && E_2 \ar[dl]
 \ar[r] & \dots  \ar[r] & E_{m-1} \ar[rr] && E_m \ar[dl] \\
& A_1 \ar@{-->}[ul] && A_2 \ar@{-->}[ul] &&&& A_m \ar@{-->}[ul]
}
= E
\end{equation}
with \emph{HN-factors} $A_i \in \sli(\phi_i)$ for all $i$.
\end{itemize}
\end{definition}
For each non-zero object $0 \neq E \in \D$, we define two real numbers by
$\phi_{\sigma}^+(E):=\phi_1$ and $\phi_{\sigma}^-(E):=\phi_m$ where $\phi_1$ and
$\phi_m$ are determined by the axiom (d).
Nonzero objects in $\sli(\phi)$ are called {\it semistable of phase $\phi$} and simple objects
in $\sli(\phi)$ are called {\it stable of phase $\phi$}.
For a non-zero object $E \in \D$ with extension factors $A_1,\dots,A_m$
given by axiom (d), define the mass of $E$ by
\[
m_{\sigma}(E):=\sum_{i=1}^m |Z(A_i)|.
\]
Following \cite{KoSo}, we assume an additional condition, called the {\it support property}.
For a stability condition $\sigma = (Z,\sli)$,
we introduce the set of semistable classes
$\ss(\sigma) \subset K(\D)$ by
\begin{equation}\label{eq:ss}
\ss(\sigma) :=\{\,\alpha \in
K(\D)\,\vert\,\text{there is a semistable object }
E \in \D \text{ such that } [E] = \alpha\,\}.
\end{equation}
Let $\norm{\,\cdot\,}$ be some norm on $K(\D) \otimes \R$.
A stability condition $\sigma=(Z,\sli)$
satisfies the support property if there is a some constant $C=C_\sigma >0$ such that
\begin{equation}\label{eq:supp}
C \cdot |{Z(\alpha)}| > \norm{\alpha}
\end{equation}
for all $\alpha \in \ss(\sigma)$.
Let $\Stab\D$ be the set of all stability conditions on $\D$ satisfying
the support property. We define the distance
\begin{gather}\label{eq:d}
    d \colon \Stab\D\times \Stab\D\to [0,\infty]
\end{gather}
by
\begin{equation}
\label{eq:distance}
d(\sigma,\varsigma):= \sup_{0 \neq E \in \D}\left\{\,
|\phi_{\sigma}^-(E) - \phi_{\varsigma}^-(E)|\,,\,
|\phi_{\sigma}^+(E) - \phi_{\varsigma}^+(E)|\,,\,
\left|\log \frac{m_{\sigma}(E)}{m_{\varsigma}(E)}  \right|
\right\}.
\end{equation}
Under the topology induced by the distance $d$ on $\Stab\D$,
Bridgeland \cite{B1} showed
the following crucial theorem.
\begin{theorem}[\cite{B1}, Theorem 1.2]
\label{thm:localiso}
The projection map of taking central charges
\begin{gather*}\begin{array}{ccc}
\mathcal{Z}\colon \Stab\D &\longrightarrow& \Hom(K(\D),\CC),\\
    \quad(Z,\hh{P})& \mapsto& Z
\end{array}\end{gather*}
is a local homeomorphism of topological spaces. In particular,
$\mathcal{Z}$ induces a complex structure on $\Stab\D$.
\end{theorem}

Here we consider the case when the Grothendieck group $K(\D)$ is of
finite rank over $\ZZ$. However in \cite{B1},
he also deals with the case when $K(\D)$ is of infinite rank. In that case, one should consider the set of all {\it locally-finite}
stability conditions $\Stab\D$.

\begin{definition}\cite[Definition 5.7]{B1}
A slicing $\sli$ of a triangulated category $\D$ is locally-finite if there
exists a real number $\epsilon>0$ such that for all $\phi\in\RR$ the quasi-abelian category
$\sli(\phi-\epsilon,\phi+\epsilon)$ is of finite length.
A stability condition is locally-finite if the corresponding slicing is.
\end{definition}

On the space of all stability conditions $\Stab\D$,
we can define two group actions commuting with each other.
The first one is the natural $\CC$ action
\[
    s \cdot (Z,\hh{P})=(Z \cdot e^{-\mathbf{i} \pi s},\hh{P}_{-\Re(s)}),
\]
where $\hh{P}_x(\phi)=\hh{P}(\phi+x)$.
There is also a natural action on $\Stab\D$ induced by $\Aut\D$, namely:
$$\Phi  (Z,\hh{P})=\big(Z \circ \Phi^{-1}, \Phi (\hh{P}) \big).$$

%=========================================================
\subsection{$q$-stability conditions}
%=========================================================
\begin{con}
Let $\D_{\XX}$ be a triangulated category with a distinguished auto-equivalence
$$\XX \colon \D_{\XX} \to \D_{\XX}.$$
Here $\XX$ is not necessarily the Serre functor.
We will write $E[l \XX]$ instead of $\XX^l(E)$ for
$l \in \ZZ $ and $E \in \D_{\XX}$.
Set
\begin{gather}\label{eq:R}
    \Zq=\ZZ[q^{\pm 1}]
\end{gather}
and define the $\Zq$-action on $K(\D_{\XX})$ by
\begin{gather}\label{eq:R-structure}
    q^l \cdot [E] := [E[l \XX]].
\end{gather}
Then $K(\D_{\XX})$ has an $\Zq$-module structure.
Moreover, when we consider such a category $\D_\XX$,
the auto-equivalence group $\Aut\D_\XX$ will only consists of
those that commute with $\XX$.
\end{con}

\begin{definition}\label{def:Xstab}
A {\it $q$-stability condition} $(\sigma,s)$
consists of a (Bridgeland) stability condition
$\sigma=(Z,\sli)$ on $\D_\XX$ and a complex number $s \in \CC$
satisfying
\begin{equation}
\label{eq:X=s}
    \XX ( \sigma)=s \cdot \sigma.
\end{equation}
We may write $\sigma[\XX]$ for $\XX(\sigma)$.
\end{definition}

\begin{remark}
Equation \eqref{eq:X=s} has been considered by Toda \cite{To}
to study the orbit point, known as the Gepner point, of the orbitfold $\CC\backslash\Stab\D/\Aut$.
We will impose the following condition (Assumption~\ref{assumption:R}) on our triangulated category $\D_\XX$,
which means $\Stab\D_\XX$ is infinite dimensional.
And the equation \eqref{eq:X=s} reduces the dimension of the stability spaces.
See \cite[\S~1.2]{Q3} for further discussion.
\end{remark}

In the rest of this paper, we assume following.
\begin{assumption}
\label{assumption:R}
The Grothendieck group $\Grot (\D_{\XX})$ is free of finite rank
over $\Zq$, i.e. $\Grot(\D_{\XX} )\cong \Zq^{\oplus n}$ for some $n $.
We will call such a category $\D_\XX$ an $\XX$-category.
\end{assumption}

For a fixed complex number $s \in \CC$, consider the specialization
\[
    q_s\colon\CC[q,q^{-1}]\to\CC,\quad q\mapsto e^{\bi \pi s}.
\]
Denote by $\CC_s$ the complex numbers with the $\Zq$-module structure through the specialization $q_s$.
To spell out the conditions for the equation \eqref{eq:X=s}, we have the following
equivalent conditions of \eqref{eq:X=s}.

\begin{definition}\label{def:xstab}
A $q$-stability condition $(\sigma,s)$ consists of
a (Bridgeland) stability condition $\sigma=(Z,\sli)$ on $\D_\XX$ and
a complex number $s \in \CC$
satisfying the following two more conditions:
\begin{itemize}
\item[(e)]
the slicing satisfies $\sli(\phi + \Re( s)) = \sli(\phi)[\XX]$
for all $\phi \in \R$,
\item[(f)]
the central charge $Z\colon K(\D_{\XX}) \to \CC_s$ is $\Zq$-linear;
\[
Z \in \Hom_\Zq(K(\D_{\XX}),\CC_s).
\]
\end{itemize}
\end{definition}

Similar to usual stability conditions, we consider the support property as follows.

\begin{definition}\label{def:support}
A $q$-stability condition $(\sigma,s)$ satisfies the {\it $q$-support property} if
there is some lattice $\Gamma:=\ZZ^n \subset K(\D_{\XX})$ satisfying
$\Gamma \otimes_{\ZZ}\Zq \cong K(\D_{\XX})$ and a subset
\begin{gather}\label{eq:sshat}
    \widehat{\ss}(\sigma) \subset \{ \alpha \in K(\D_{\XX}) \,|\,
        \alpha= \sum_{j=0}^l q^j  \alpha_j \, (\alpha_j \in \Gamma) \}
\end{gather}
such that
\begin{enumerate}
\item
the set of semistable classes $\ss(\sigma)$ is given by
\[
   \ss(\sigma)=\bigcup_{k \in \ZZ}q^k \cdot \widehat{\ss}(\sigma).
\]
\item
for some norm $\norm{\,\cdot\,}$ on a finite dimensional vector space $\Gamma \otimes_{\ZZ}\R$,
there is some constant $C>0$ such that \eqref{eq:supp} holds
for all $\alpha \in  \widehat{\ss}(\sigma)$ (and hence all $\alpha\in\ss(\sigma)$),
\Note{where $$\norm{ \sum_{j=0}^l q^j  \alpha_j }\colon= \sum_{j=0}^l |e^{ \bi \pi j s}|\cdot
      \norm{\alpha_j}. $$}
\item $\XX$-$\Hom$-bounded: for any semistable object $E$ with $[E]\in\Gamma$, exists $N_0$ such that for any stable object $F$ with $[F]\in\Gamma$,
\begin{equation}\label{eq:Hom=0}
\Hom(E,F[k\XX])=0=\Hom(F,E[k\XX])
\end{equation}
when $|k|>N_0$.
\end{enumerate}
\end{definition}

%\begin{definition}
%\label{def:support}
%Fix a complex number $s$.
%Let $\norm{\,\cdot\,}_s$ be some norm on $K(\D_{\XX})
%\otimes \R$ satisfying
%\begin{equation*}
%\norm{q^n \cdot \alpha}_s=|e^{ \bi \pi s}|^n \cdot \norm{\alpha}_s  \quad(n \in \ZZ).
%\end{equation*}
%A $q$-stability condition $(\sigma,s)$ satisfies the {\it support property} if
%\Note{
%\begin{itemize}
%\item the set of semistable classes $\ss(\sigma)$
%equals
%\[
 %   R\otimes\widehat{\ss}(\sigma)\colon=\{q^k\cdot\alpha\mid \alpha\in\widehat{\ss}(\sigma), k\in\ZZ\}
%\]
%for some $\ZZ^n\cong\widehat{\ss}(\sigma)\subset K\D_\XX$ and
%\item there is some constant $C>0$ such that \eqref{eq:supp} holds
%for all $\alpha \in \ss(\sigma)$.
%\end{itemize}
%}
%\end{definition}

%\begin{definition}
%\label{def:locally-discrete}
%Fix $0<\epsilon<1 \slash 2$.
%A $q$-stability condition $(\sigma,s)$ is {\it locally-discrete}
%if for each $\phi \in \RR$, the image of
%$Z \colon K(\D_{\XX})\to \CC_s$ on $\sli((\phi-\epsilon,\phi+\epsilon))$
%is discrete.
%\end{definition}

\begin{remark}\label{rem:changeB}(Change of Basis)
Given any other lattice $\Gamma'\subset K(\D_{\XX})$ satisfying $\Gamma \otimes_{\ZZ}\Zq \cong K(\D_{\XX})$,
then there is another subset
\[
    \widehat{\ss'}(\sigma) \subset \{ \alpha \in K(\D_{\XX}) \,|\,
    \alpha= \sum_{j=0}^l q^j  \alpha_j \, (\alpha_j \in \Gamma') \}
\]
such that Condition~$2^\circ$ and $3^\circ$ holds for $\widehat{\ss'}(\sigma)$ and $\Gamma'$ respectively.
More precisely,
\begin{itemize}
\item
The second condition holds for all $\alpha\in\ss(\sigma)$ and will effect neither by
the choices of $\Gamma'$ nor by $\widehat{ss}'(\sigma)$.
\item $\XX$-$\Hom$-boundedness holds since the Hom-vanishing property preserves under (iterated) extension.
\end{itemize}
In other words, the $q$-support property is intrinsic
(that does not depend on the choice of the lattice $\Gamma$).
\end{remark}

Denote by $\QStab_s\D_\XX$ the set of all $q$-stability conditions
satisfying the $q$-support property and with fixed $s$.
Since the space $\QStab_s\D_\XX$ is
a subset of the space of usual stability conditions
$\Stab(\D_{\XX})$ on $\D_{\XX}$,
the distance $d$ (\ref{eq:distance}) on $\Stab\D$
induces a topology on $\QStab_s\D_\XX$.
We proceed to show the analogue result of Theorem \ref{thm:localiso}.

\begin{theorem}\label{thm:localiso2}
The projection map of taking central charges
\begin{gather}\label{eq:Zs}\begin{array}{ccc}
\mathcal{Z}_s \colon \QStab_s\D_{\XX} &\longrightarrow& \Hom_\Zq(K(\D_{\XX}),\CC_s),\\
    \quad((Z,\hh{P}),s)& \mapsto& Z
\end{array}\end{gather}
is a local homeomorphism of topological spaces. In particular,
$\mathcal{Z}_s$ induces a complex structure on $\QStab_s\D_{\XX}$.
\end{theorem}
Next subsection is devoted to the proof of this theorem.
%=========================================================
\subsection{Proof of the deformation theorem}
%=========================================================
%\Note{seems locally discrete is not necessary.}
%\begin{definition}
%\label{def:locally-discrete}
%Fix $0<\epsilon<1 \slash 2$.
%A $q$-stability condition $(\sigma,s)$ is {\it locally-discrete}
%if for each $\phi \in \RR$, the image of the map
%$Z \colon K(\D_{\XX}) \to \CC_s $ for objects in
%$\sli((\phi-\epsilon,\phi+\epsilon))$ is discrete.
%\end{definition}
We divide the proof into five steps,
recall that we fix the complex number $s$.
The outline is that we will first prove the locally finiteness
and adapt Bridgeland's deformation strategy to get a stability conditions,
which we need to show that it is a $q$-stability condition
satisfies/preserves the corresponding properties.
%The adaptation of proof is not always straightforward since the rank of the Grothendieck group
%over $\ZZ$ is not finite.

\subsubsection*{\bf{Step I}}The $q$-support property implies locally-finiteness.
\begin{lemma}
If a $q$-stability condition $(\sigma,s)$ satisfies the $q$-support property
and $s\ne0$, then it is locally-finite.
%\begin{itemize}
%  \item If $(\sigma,s)$ satisfies the support property and $s \in \RR \setminus \{0\}$,
%  then $(\sigma,s)$ is locally-discrete.
%  \item If $(\sigma,s)$ is locally-discrete, then it is locally-finite.
%\end{itemize}
\end{lemma}
\begin{proof}
%\Note{
Suppose not, then without loss of generality
we can assume that there is an infinite chain of subobjects
$\cdots \subset E_m\subset\cdots \subset E_1=E$
in the quasi-abelian category $\sli(\phi-\epsilon,\phi+\epsilon)$
for some $\phi$ and $\epsilon<1/2$.
As in the proof of \cite[Lemma.~4.4]{B2},
the norms $|Z(E_m)|$ of central charges are bounded, say by $K>0$.
Thus the following set
\[\begin{array}{rl}
    \{ \,\alp \in \ss(\sigma) \mid& |Z(\alp)| <K,
    \;\exists M \in \sli(\phi-\epsilon,\phi+\epsilon)
    \text{ s.t.}\\ & [M]=\alp,\;\Hom(M,E)\ne0     \},
\end{array}\]
denoted by $\Lambda_{\sigma}$, is infinite.
Taking their HN-factors if necessary,
we may assume that both $E$ and $M$ are semisimple when considering $\Lambda_{\sigma}$.

%Since \[\ss(\sigma)=\bigcup_{k \in \ZZ}q^k \cdot \widehat{\ss}(\sigma).\]
By Condition~$1^\circ$ of $q$-support property,
any $\alpha\in\Lambda_{\sigma}$ equals $q^k\cdot\widehat{\alp}$
for some $\widehat{\alp}\in\widehat{\ss}$.
Thus $\Lambda_{\sigma}=\bigcup_{k} \Lambda_k$, where $\Lambda_k$ is
\[\begin{array}{rl}
    \{ \,\widehat{\alp} \in \widehat{\ss}(\sigma)\mid&
    |e^{k\mathbf{i}\pi s}\cdot Z(\widehat{\alp})| <K,
    \;\exists M \in \sli(\phi-\epsilon,\phi+\epsilon)\text{ s.t.}\\
    & [M]=q^k\widehat{\alp},\;\Hom(M,E)\ne0  \}.
\end{array}\]
Note that $[E[l\XX]]\in\widehat{\ss}(\sigma)$ for some $l\in\ZZ$.
Then by Condition~$3^\circ$ (i.e. $\XX\Hom$-bounded) of $q$-support property,
only finitely many $\Lambda_k$ are not empty.
Therefore, there exists $k_0\in\ZZ$ such that $\Lambda_{k_0}$ is infinite.

On the other hand,
by support property \eqref{eq:supp},
\[
    \{ \,\widehat{\alp} \in \widehat{\ss}(\sigma)\,\mid\, \,    |Z(\widehat{\alp})| <K'\,  \}
\]
is finite for any $K'>0$.
Taking $K'=|e^{-k\mathbf{i}\pi s}| \cdot K$, then $\Lambda_k$ is finite for any $k$,
which is a contradiction.
%}
\end{proof}

\subsubsection*{\bf{Step II}}
Adopting the proof of Bridgeland \cite[Theorem 1.2]{B1},
we obtain the following deformation statement as follows.

\begin{corollary}\label{cor:B}
Let $(\sigma,s)$ be a $q$-stability condition satisfying $q$-support property,
where $\sigma=(Z,\hh{P})$.
Then there exists $0<\epsilon<1/8$, such that for any $W\in\Hom_\Zq(K(\D_\XX),\CC_s)$
satisfying that
\begin{gather}\label{eq:W-Z}
    | W(E)-Z(E) | < \sin(\epsilon\pi)|Z(E)|
\end{gather}
holds for any $\sigma$-stable object $E$, there exists a slicing $\hh{Q}$
so that $\varsigma=(W,\hh{Q})$ forms a stability conditions on $\D_\XX$ with
$d(\sigma,\varsigma)<\epsilon$.
Here the distance $d$ is defined as in \eqref{eq:d}.
\end{corollary}

For the later proof, let us sketch the construction of the slicing $\hh{Q}$
(from \cite[Theorem 1.2]{B1}),
which is the key of the proof of the corollary above.
Fix $\epsilon$ so that any (quasi-abelian) subcategory $\hh{P}(t-\eta,t+\eta)$ is of finite length
(i.e. both artinian and noetherian) for any $\eta\le\epsilon$.
%A thin subcategory in $\D_\XX$, with respect to $\hh{P}$, is the full subcategory of the form
%$\hh{P}(a,b)$ for some interval with $0<b-a<1-2\epsilon$.
Then $\hh{Q}(\psi)$ is the full additive subcategories of $\D_\XX$ consisting of
the zero objects together with those object $E$, which is
$W$-semistable with phase $\psi$ in the subcategory
\begin{gather}\label{eq:hhP}
    \hh{P}( \psi-\epsilon,\psi+\epsilon ).
\end{gather}
Note that $E$ is in fact $W$-semistable in any subcategory $\hh{P}(I)$
if the interval $I$ contains $( \psi-\epsilon,\psi+\epsilon )$.

\subsubsection*{\bf{Step III}}
Now we prove that the deformed stability condition $\varsigma$ in Corollary~\ref{cor:B}
is a $q$-stability condition (together with $s$).

\begin{lemma}
The slicing $\hh{Q}$ satisfies Condition~(e) in Definition~\ref{def:xstab},
thus $(\varsigma,s)$ is a $q$-stability condition.
\end{lemma}
\begin{proof}
We only need to prove that, if $E$ is in $W(\psi)$, then $E[\XX]$ is $W(\psi+\Re(s))$.
By construction, $E$ is $W$-semistable with phase $\psi$ in $\hh{P}( \psi-\epsilon,\psi+\epsilon )$.
Now suppose that
$E[\XX]$ is not $W$-semistable in $\hh{P}( \psi+\Re(s)-\epsilon,\psi+\Re(s)+\epsilon )$.
Then there is a short exact sequence
\[
    0\to A\to E[\XX]\to B\to 0
\]
in $\hh{P}( \psi+\Re(s)-\epsilon,\psi+\Re(s)+\epsilon )$
such that $\phi_W(A)>\phi_W(E[\XX])>\phi_W(B)$.
Since $\hh{P}$ satisfies Condition~(e) in Definition~\ref{def:xstab},
here is a short exact sequence
\[
    0\to A[-\XX]\to E\to B[-\XX]\to 0
\]
in $\hh{P}( \psi-\epsilon,\psi+\epsilon )$
such that $\phi_W(A[-\XX])>\phi_W(E)>\phi_W(B[-\XX])$.
This contradicts to the fact that $E$ is $W$-semistable in $\hh{P}( \psi-\epsilon,\psi+\epsilon )$,
that finishes the proof.
\end{proof}

\subsubsection*{\bf{Step IV}}Next we check the $q$-support property is preserved under deformation.
\begin{lemma}
The $q$-stability condition $(\varsigma,s)$ in Corollary~\ref{cor:B} satisfies the $q$-support property.
\end{lemma}
\begin{proof}
Let $\widehat{\ss}(\varsigma)$ be any set in the form of \eqref{eq:sshat}
satisfying the first condition in \eqref{def:support}, with respect to the lattice $\Gamma$.
This can be done by choosing one representative $q^{k_0}\alpha$ in each subset
$\{ q^k\alpha \mid k\in\ZZ \}\subset\ss(\varsigma)$, where $k_0$ is bigger enough.

Now let $\alpha=[E]\in\widehat{\ss}(\varsigma)$ with $\varsigma$-semistable $E$.
Consider the Harder-Narasimhan filtration \eqref{eq:HN} of $E$ with respect to $\sigma$.
So $[A_i]\in\ss(\sigma)$ and
\[
    \alpha=[E]=\sum_{i}^m [A_i].
\]
By \eqref{eq:W-Z}, we have
\[
    |W(\alpha)|>(1-\sin(\epsilon\pi))|W(\alpha)|.
\]
Moreover, since $E$ is in $\hh{P}( \psi-\epsilon,\psi+\epsilon )$ in \eqref{eq:hhP},
the phases of $A_i$ and $E$, w.r.t. $\sigma$ (or $\hh{P}$), is within an open interval of length $2\epsilon$.
Thus
\[
    |Z(E)|=| \sum_{i}^m Z(A_i) |> \sum_{i}^m \cos(2\epsilon) | Z(A_i) |
\]
Combining the calculations above and the fact that $[A_i]$ satisfies \eqref{eq:supp}
(as $\sigma$ satisfies support property), we have
\[\begin{array}{rl}
|W(\alpha)|&>(1-\sin(\epsilon\pi))|Z(\alpha)|\vspace{1ex}\\
&>(1-\sin(\epsilon\pi)) \cdot \cos(2\epsilon)  \displaystyle\sum_{i}^m| Z(A_i) |\\
&>(1-\sin(\epsilon\pi)) \cdot  \cos(2\epsilon) \cdot  C_\sigma^{-1}\displaystyle\sum_{i}^m| \norm{[A_i]}\\
&>(1-\sin(\epsilon\pi)) \cdot  \cos(2\epsilon) \cdot  C_\sigma^{-1}  \norm{\displaystyle\sum_{i}^m| [A_i]}\vspace{1ex}\\
&>(1-\sin(\epsilon\pi)) \cdot  \cos(2\epsilon) \cdot  C_\sigma^{-1}  \norm{\alpha}.
\end{array}\]
Thus, we can take $C_\varsigma=(1-\sin(\epsilon\pi))^{-1} \cdot  \cos(2\epsilon)^{-1} \cdot C_\sigma$ so that
$\varsigma$ satisfies the second condition in Definition~\ref{def:support}.

Finally, the third condition holds as explained in Remark~\ref{rem:changeB}.
\end{proof}

This completes the proof.
%\subsubsection*{\bf{Step V}}
%Till now, we have shown that one can deform $q$-stability conditions in $\QStab_s\D_\XX$
%and hence \eqref{eq:Zs} is a local homeomophism onto an open subspace of $\Hom_\Zq(K(\D_{\XX}),\CC_s)$.
%We finished the proof of Theorem~\ref{thm:localiso2} by showing the fullness
%(\cite[Definition 4.2]{B2}), that is, $\QStab_s\D_\XX$ is of full rank (i.e. complex dimension of $n$).
%Note that, in the usual setting, i.e. when the rank of the Grothendieck group is finite
%\cite[Proposition B.4]{BM} showed that
%the support property is equivalent to fullness.
%In our setting, this more or less still the case.
%Namely, we only need to show that
%the induced metric $\norm{\cdot}^\vee$ on $\Hom_\Zq(K(\D_{\XX}),\CC_s)$ dominates
%the $\sigma$-metric $\norm{\cdot}_\sigma$ on $\Hom_\Zq(K(\D_{\XX}),\CC_s)$, defined as
%\[
%    \norm{ W }_\sigma\colon=\sup\left\{
%        \frac{ W(E) }{ Z(E) } \mid \text{ $E$ is $\sigma$-semistable}
%    \right\}
%\]
%for any $\sigma\in\QStab_s\D_\XX$.
%This follows from the second condition of $q$-support property,
%where the calculation is exactly the same as the first part of proof of \cite[Proposition B.4]{BM}.
%=========================================================
\subsection{Gluing $q$-stability conditions}
%=========================================================
By Assumption~\ref{assumption:R}, $\QStab_s\D_\XX$ is of (complex) dimension $n$.

With respect to the original topology of Bridgeland,
$\QStab_s\D_\XX$ are in different connected components of $\Stab\D_\XX$
for different $s\in\CC$.

\begin{lemma}
Let $(\sigma_i,s_i)$ be $q$-stability conditions with $s_1\neq s_2$.
Then they are in different connected components of $\Stab_s\D_\XX$.
\end{lemma}
\begin{proof}
Recall the distance $d$ on $\Stab\D_\XX$ is defined in \eqref{eq:distance}.
We have
\[
    \phi_{\sigma_1}^\pm(E[k\XX]) - \phi_{\sigma_2}^\pm(E[k\XX])=
    \phi_{\sigma_1}^\pm(E) - \phi_{\sigma_2}^\pm(E)+k(\Re(s_1)-\Re(s_2))
\]
for any $k\in\ZZ$ by (e) of Definition~\ref{def:xstab}.
Thus if $\Re(s_1)\ne\Re(s_2)$, then we have
\[
    d(\sigma_1,\sigma_2)\ge
    \sup_{k\in\ZZ} \left\{\,
    \left|
        \phi_{\sigma_1}^\pm(E) - \phi_{\sigma_2}^\pm(E)+k(\Re(s_1)-\Re(s_2))
    \right|\right\}=\infty.
\]
Similarly, we have
\[
    \log \frac{m_{\sigma_1}(E[k\XX])}{m_{\sigma_2}(E[k\XX])}
    =\log \frac{m_{\sigma_1}(E)}{m_{\sigma_2}(E)}+k(\Im(s_1)-\Im(s_2))
\]
for any $k\in\ZZ$ by (f) of Definition~\ref{def:xstab}.
Thus if $\Im(s_1)\ne\Im(s_2)$, then we have
\[
    d(\sigma_1,\sigma_2)\ge
    \sup_{k\in\ZZ} \left\{\,
    \left|
        \log \frac{m_{\sigma_1}(E)}{m_{\sigma_2}(E)}+k(\Im(s_1)-\Im(s_2))
    \right|\right\}=\infty.
\]
\end{proof}

So the question now is if/how we can gluing different connected components $\QStab_s\D_\XX$ in $\Stab\D_\XX$ together.
We hope that one can deform along the $s$ direction to reveal the relations between these complex manifolds $\QStab_s\D_\XX$.
Set
\[
    \QStab\D_\XX:=\bigcup_{s\in\CC}\QStab_s\D_\XX.
\]

\begin{conjecture}\label{conj:xstab}
$\QStab\D_\XX$ admits the structure of
a complex manifold of dimension $n+1$
and the projection map
\[
\pi \colon \QStab\D_\XX \to \CC,\quad (\sigma,s) \mapsto s
\]
is holomorphic.
%In addition, there is some number $r \in \RR$
%such that $\QStab(\D_{\XX})$ on
%$\CC_{>r}:=\{s \in \CC\,|\, \Re (s) >r\}$ is
%a trivial bundle, namely
%\[
%\pi^{-1}(\CC_{>r})
%\cong \QStab_{s_0}(\D_{\XX}) \times \CC_{>r}
%\]
%where $s_0$ is some fixed complex number in $\CC_{>r}$.
\end{conjecture}
A partial answer to this conjecture is provided in Theorem~\ref{thm:manifold},
that an open subspace of $\QStab\D_\XX$ consisting of `induced' $q$-stability conditions
does glue together.
Moreover, such a subspace for type $A_2$ quiver is calculated in Section~\ref{sec:A2}.
These induced $q$-stability conditions will be identified with
multi-valued quadratic differentials in the surface case in the sequel \cite{IQ2}
(and hence are the most interesting ones as far as we are concerned).

%=========================================================
\section{Reduction}
%=========================================================
In this section, we show that under some conditions,
the space of $q$-stability conditions with $s \in \ZZ$ coincides
with the space of usual stability conditions (on a triangulated category
with finite rank Grothendieck group).
First we recall the notion of orbit categories.
Let $\D$ be a triangulated category with a functor %\Note{the functor (or auto-equivalence)}
$\Phi \colon \D \to \D$, the orbit category $\D \slash \Phi$ is
defined to be the category whose objects are the same as $\D$
and whose morphism spaces are given by
\[
\Hom_{\D \slash \Phi}(E,F):=\bigoplus_{k \in \ZZ}\Hom_{\D}(E,\Phi^k(F)).
\]
As in the previous section, let
$\D_{\XX}$ be a triangulated category
with a distinguished auto-equivalence
$\XX \colon \D_{\XX} \to \D_{\XX}$ satisfying
Assumption \ref{assumption:R}.

\begin{definition}
Let $N\ge1$ be an integer.
The orbit quotient $$\D_N=\D_{\XX} \sslash [\XX-N]$$ is defined to be the triangulated hull of the orbit category $\D_{\XX} \slash [\XX-N]$.
It is {\it $N$-reductive} if
\begin{itemize}
\item the quotient functor $\pi_N \colon  \D_{\XX} \to \D_N$ is exact,
\item the Grothendieck group of $\D_N$ is
free of finite rank, i.e. $K(\D_N) \cong \ZZ^{\oplus n}$ and the induced
$\Zq$-linear map
\[
[\pi_N] \colon K(\D_{\XX}) \to K(\D_N)
\]
is a surjection given by sending $q \mapsto (-1)^N$.
\end{itemize}
\end{definition}

%\begin{remark}
In our motivating examples, $\D_\XX$ is constructed with dg-structures,
which provides a dg enhancement of $\D_{\XX} \sslash [\XX-N]$.
%\end{remark}

\begin{construction}
Assume that $\D_{\XX}$ is $N$-reductive and let $(\sigma,N)$ be an
$q$-stability condition on $\D_{\XX}$ with $\sigma=(Z,\sli)$.
We define a stability conditions $\sigma_N=(Z_N,\sli_N)$ on $\D_N$ as follows
\begin{itemize}
  \item $\sli_N(\phi):=\sli(\phi)$;
  \item By Condition (f) in Definition~\ref{def:xstab},
  $Z\colon K(\D_{\XX})\to\CC$ factors through $[\pi_N]$
  and thus we obtain a group homomorphism $Z_N\colon K(\D_N)\to\CC$
  satisfying $Z=Z_N\circ[\pi_N]$.
\end{itemize}
\end{construction}

\begin{lemma}\label{lem:above}
If $\D_N$ is $N$-reductive, then
$\sigma_N$ is a stability condition in $\Stab\D_N$.
\end{lemma}
\begin{proof}
Since $(\sigma,N)$ satisfies Condition (e) in Definition~\ref{def:xstab},
$\sli$ on $\D_N$ is well-defined and satisfies Condition (b) in Definition~\ref{def:stab}.
Since $\pi_N$ is exact, $\sli$ satisfies Condition (c), (d) in Definition~\ref{def:stab}.
and thus a slicing.
Thus, the lemma follows.
\end{proof}

For a $q$-stability condition $\sigma$, define a real number $L(\sigma)$ by
\[
L(\sigma):=\inf\left\{\frac{|Z(E)|}{\norm{[E]}} \,\middle|\,
[E] \in \widehat{\ss}(\sigma)  \ \right\},
\]
where $\norm{[E]}$ is the norm in Definition \ref{def:support}.
Then we note that
the support property is equivalent to the condition $L(\sigma)>0$.

\begin{proposition}\label{pp:above}
Let $\{\sigma_k\}_{k \ge 1}$ is a sequence of
$q$-stability conditions in $\QStab_s\D_{\XX}$ satisfying
\begin{itemize}
\item central charges $Z_k$ converges as $k \to \infty$ in
$\Hom_{R}(K(\D_{\XX}),\CC_s)$,
\item the slicing of $\sigma_k$ converges in the space $\operatorname{Slice}\D_{\XX}$ of slicings on $\D_{\XX}$,
where the metric on $\operatorname{Slice}\D_{\XX}$ is given by
\begin{equation}
\label{eq:d-slicing}
d(\hua{P},\hua{Q}):= \sup_{0 \neq E \in \D}\left\{
|\phi_{\hua{P}}^-(E) - \phi_{\hua{Q}}^-(E)|,
|\phi_{\hua{P}}^+(E) - \phi_{\hua{Q}}^+(E)|  \right\}.
\end{equation}
\item there is a uniform constant $L>0$ such that
\[
L(\sigma_k) \ge L
\]
for all $k $.
\end{itemize}
Then the sequence $\{\sigma_k \}_{k \ge 1}$ converges to some
stability condition $\sigma_{\infty} \in \QStab_s\D_{\XX}$ as $k \to \infty$.
\end{proposition}
\begin{proof}
We show that for any real small number $\epsilon >0$,
the limit
$Z_{\infty}:=\lim_{k\to \infty}Z_k$ satisfies
\[
|Z_{\infty}(E) - Z_k(E)|<\epsilon \cdot|Z_k(E)|
\]
for sufficiently large $k$ and
all $[E] \in \widehat{\ss}(\sigma_k)$. Then by Theorem~\ref{thm:localiso2},
there exists a unique stability condition $\sigma_{\infty}$ with the central charge $Z_{\infty}$
satisfying $d(\sigma_{\infty},\sigma_k)<\epsilon$ and hence as required.

Consider the operator norm $ \norm{\,\cdot\,}$ on
$\Hom_{R}(K(\D_{\XX}),\CC_s)$ defined by
\[
\norm{W}:=\sup_{0 \neq E \in \D_{\XX}} \frac{|W(E)|}{\norm{[E]}}.
\]
Then since $Z_k \to Z_{\infty}$ as $k \to \infty$,
we have
\[
\norm{Z_{\infty}-Z_k}<\epsilon \cdot L
\]
for sufficiently large $k$. This implies
\[
|Z_{\infty}(E)-Z_k(E)|<\epsilon  \cdot L\cdot \norm{[E]}
\]
for all $0 \neq E \in \D_{\XX}$. On the other hand,
from the condition $L(\sigma_k )\ge L$, we have
\[
|Z_k(E)| \ge L \cdot \norm{[E]}
\]
for all $[E] \in \widehat{\ss}(\sigma_k)$.
\end{proof}

Combining Lemma~\ref{lem:above} and Proposition~\ref{pp:above},
we have the following,
which is one of the motivations that we introduce $q$-stability conditions.

\begin{theorem}
\label{thm:reduction}
If $\D_{\XX}$ is an $N$-reductive, then
there is a canonical injection of complex manifolds
\[
    \iota_N \colon \QStab_N(\D_{\XX}) \to \Stab\D_N,
\]
whose image is open and closed.
\end{theorem}
\begin{proof}
We show the closedness of the image of $\iota_N$ and the other part is straightforward.
Take a  convergent sequence $\{\sigma_n\}_{n \ge 1}$ in $\Stab\D_N$
with the limit $\sigma_{\infty} \in \Stab\D_N$
and assume that $\sigma_n=\iota_{N}(\tilde{\sigma}_n)$
for $\tilde{\sigma}_n \in  \QStab_N(\D_{\XX})$. We show that there is some
$\tilde{\sigma}_{\infty} \in \QStab_N(\D_{\XX})$ such that $\iota_N(\tilde{\sigma}_{\infty})=\sigma_{\infty}$.
We check the sequence $\{\tilde{\sigma}_n\}_{n \ge 1}$ satisfies the conditions of Proposition \ref{pp:above}.

The first two conditions follows by direct checking as $\sigma_n$ and $\tilde{\sigma}_n$ as their central charges and slicings
are essentially the same.
Next we consider the third condition. Since
the sequence $\{L(\tilde{\sigma}_n)\}_{n \ge 1}$
converges to some positive number $L(\tilde{\sigma}_{\infty} )>0$, there is uniform constant $L>0$
such that $L(\tilde{\sigma}_n)>L$. Then the second condition
also holds since $L(\sigma_n)=L(\tilde{\sigma}_n)$ by definition.
\end{proof}

In the next section, we introduce a special type of $q$-stability conditions,
the induced $q$-stability conditions, which in many cases provide the existence of $q$-stability conditions.

%=========================================================
\section{Induction}\label{sec:ind}
%=========================================================
\subsection{$\xx$-baric hearts and induced pre $q$-stability conditions}\label{sec:QSC}
%=========================================================
In this section, we introduce $\XX$-baric heart (a triangulated category)
in an $\XX$-category $\D_\XX$ as the $\XX$-analogue
of the usual heart (an abelian category) of a triangulated category.
Note that this is a special case of the baric structure
studied by Achar-Treumann \cite{AT} (see also \cite{FM}).

\begin{definition}
%Let $\D_{\XX}$ be a Calabi-Yau-$\XX$ category.
An \emph{$\XX$-baric heart} $\D_{\infty} \subset \D_{\XX}$
is a full triangulated subcategory of $\D_{\XX}$ satisfying the following conditions:
\begin{itemize}
\item[(1)]if $k_1 > k_2$ and $A_i \in \D_{\infty}[k_i\XX]\,(i =1,2)$,
then $\Hom_{\D_{\XX}}(A_1,A_2) = 0$,
\item[(2)]for $0 \neq E \in \D_{\XX}$,
there is a finite sequence of integers
\begin{equation*}
k_1 > k_2 > \cdots > k_m
\end{equation*}
and a collection of exact triangles
\begin{equation}\label{eq:X-HN}
0 =
\xymatrix @C=5mm{
 E_0 \ar[rr]   &&  E_1 \ar[dl] \ar[rr] && E_2 \ar[dl]
 \ar[r] & \dots  \ar[r] & E_{m-1} \ar[rr] && E_m \ar[dl] \\
& A_1 \ar@{-->}[ul] && A_2 \ar@{-->}[ul] &&&& A_m \ar@{-->}[ul]
}
= E
\end{equation}
with $A_i \in \D_{\infty}[k_i\XX]$ for all $i$.
\end{itemize}
\end{definition}

Note that by definition, classes of objects in $\D_{\infty}$
span $K(\D_{\XX})$ over $\Zq$ and we have a canonical isomorphism
\begin{gather}\label{eq:KKK}
    K(\D_{\infty}) \otimes_{\ZZ} \Zq \cong K(\D_{\XX}).
\end{gather}

The triangulated category $\D_\XX$ is Calabi-Yau-$\XX$
if $\XX$ is the Serre functor, i.e. there is a natural isomorphism:
\begin{gather}\label{eq:X}
    \XX:\Hom (X,Y)
        \xrightarrow{\sim}\Hom (Y,X[\XX])^\vee,
\end{gather}
where $V^{\vee}$ is the (graded, if $V$ is) dual space of $\bK $-vector space $V$.
For an $\XX$-baric heart $\D_{\infty}$ in a Calabi-Yau-$\XX$ category $\D_\XX$,
Condition $(1)$ can be refined as
\begin{gather}\label{eq:01}
    \Hom_{\D_{\XX}}(A_1,A_2) = 0,
\end{gather}
for $A_i\in\D_\infty[k_i\XX]$ and $k_1-k_2\notin\{0,1\}$.

Recall we have the specialization
\[
    q_s\colon\CC[q,q^{-1}]\to\CC,\quad q\mapsto e^{\mathbf{i} \pi s}.
\]

\begin{construction}\label{con:q}
Consider a triple $(\D_\infty,\ns,s)$ consists of
an $\XX$-baric heart $\D_\infty$, a (Bridgeland) stability condition $\ns=(\nz,\np)$ on $\D_\infty$
and a complex number $s$.
We construct

\begin{enumerate}
\item the additive pre-stability condition $\sadd=(Z,\padd)$ and
\item the extension pre-stability condition $\sext=(Z,\pext)$,
\end{enumerate}
where
\begin{itemize}
\item first extend $\nz$ to
\[
    Z_q\colon=\nz \otimes 1\colon K(\D_\XX)\to\CC[q,q^{-1}]
\]
via \eqref{eq:KKK} and
$$Z=q_s\circ Z_q\colon K(\D_\XX)\to\CC$$
gives a central charge function on $\D_\XX$;
\item the pre-slicing $\padd$ is defined as
\begin{equation}\label{eq:padd}
    \padd(\phi)=\add^s\np[\ZZ\XX]\colon=\add \bigoplus_{k\in\ZZ}  \np(\phi-k\Re(s))[k\XX].
\end{equation}
\item the pre-slicing $\pext$ is defined as
\begin{equation}\label{eq:pext}
    \pext(\phi)=\<\np[\ZZ\XX]\>^s\colon=\<  \np(\phi-k\Re(s))[k\XX] \>.
\end{equation}
\end{itemize}
Note that $\sigma$ does not necessary satisfy condition (d) in Definition~\ref{def:stab}
and hence may not be a stability condition.
We call such data (a central charge and collection of additive subcategories) a pre-stability condition
if they satisfy conditions (a), (b) and (c) in Definition~\ref{def:stab}.
\end{construction}
\begin{remark}
Clearly, $\padd(\phi)\subset\pext(\phi)$
although their sets of simple objects coincide.
Also note that $\sadd$ or $\sext$, as a pre-stability condition, may be induced from different triples.
\end{remark}

By construction, for any object $E\in\np(\phi), k\in\ZZ$,
    \[   Z_q(E[k\XX])=q^k\cdot m(E)\cdot e^{\mathbf{i} \pi \phi},   \]
where $m(E)\in\mathbb{R}_{>0}$.

%=========================================================
\subsection{On global dimensions of stability conditions}
%=========================================================
\begin{definition}[Global dimension]
Given a slicing $\sli$ on a triangulated category $\D$,
define the global dimension of $\sli$ by
\begin{gather}\label{eq:geq}
\gldim\sli=\sup\{ \phi_2-\phi_1 \mid
    \Hom(\sli(\phi_1),\sli(\phi_2))\neq0\}.
\end{gather}
For a stability conditions $\sigma=(Z,\sli)$ on $\D$,
its global dimension $\gldim\sigma$ is defined to be $\gldim\sli$.
\end{definition}

\begin{remark}
Given a heart $\h$ in $\D$, let $\sli$ be the
associated slicing with $\sli(\phi)=\h[\phi]$ for $\phi\in\ZZ$
and $\sli(\phi)=\emptyset$ otherwise.
Then we have\[ \gldim\sli=\gldim \h.\]
\end{remark}

\begin{lemma}
Let $\sli$ be a slicing on $\D$ with heart $\h_\phi=\sli[\phi,\phi+1)$.
Then
\[
    \mid \gldim\sli-\gldim\h_\phi \mid\leq 1.
\]
\end{lemma}

Recall that $\operatorname{Slice}\D$ the space of (locally-finite) slicings on $\D$
with the generalized metric (\cite[Lemma~6.1]{B1})
\begin{gather}\label{eq:slice d}
\begin{array}{rl}
    d(\sli,\hh{Q})&=\inf\{ \epsilon\in\mathbb{R}_{\geq0}\mid
        \hh{Q}(\phi)\subset\sli[\phi-\epsilon,\phi+\epsilon], \forall \phi\in\mathbb{R} \}\\
    &=\sup\left\{
        |\phi^+_{\hh{Q}}(E)-\phi^+_{\sli}(E)|,
        |\phi^-_{\hh{Q}}(E)-\phi^-_{\sli}(E)| \; \Big{|} 0\neq E\in\D \right\}.
\end{array}
\end{gather}
Moreover, the generalized metric on $\Stab\D$ can be defined as (\cite[Proposition~8.1]{B1})
\[
    d(\sigma_1,\sigma_2)=\sup\left\{
        |\phi^+_{\sigma_2}(E)-\phi^+_{\sigma_1}(E)|,
        |\phi^-_{\sigma_2}(E)-\phi^-_{\sigma_1}(E)| ,
        |\log\frac{m_{\sigma_2}(E)}{m_{\sigma_1}(E)}|\; \Big{|} 0\neq E\in\D \right\}.
\]

Therefore we have the following.

\begin{lemma}\label{lem:conti}
The function $\gldim \colon \operatorname{Slice}\D \to \R_{\ge 0}$ is continuous
and hence induces a continuous function
\begin{gather}\label{eq:gldim}
    \gldim\colon\Stab\D\to\mathbb{R}_{\ge 0}
\end{gather}
on $\Stab\D$.
\end{lemma}
\begin{proof}
We will use the first line of \eqref{eq:slice d} as definition
for the topology of $\operatorname{Slice}\D$.
For any $\epsilon>0$, we claim that
if $d(\sli,\hh{Q})<\epsilon/2$, then
\[\mid \gldim\sli-\gldim\hh{Q}\mid<\epsilon.\]

For any $\delta>0$, there exists $\phi_1,\phi_2\in\mathbb{R}$ such that
$\phi_2-\phi_1\in(\gldim\hh{Q}-\delta,\gldim\hh{Q}]$ and
\[
    \Hom\big( \hh{Q}(\phi_1),\hh{Q}(\phi_2) \big)\neq0.
\]
As $\hh{Q}(\phi_i)\subset\sli(\phi_i-\epsilon/2,\phi_i+\epsilon/2)$,
we have
\[
    \Hom\big( \sli(\phi_1-\epsilon/2,\phi_1+\epsilon/2),
        \sli(\phi_2-\epsilon/2,\phi_2+\epsilon/2) \big)\neq0.
\]
Thus, $$\gldim\sli\geq (\phi_2-\epsilon/2 )- (\phi_1+\epsilon/2)>\gldim\hh{Q}-\delta-\epsilon,$$
which implies that $\gldim\sli>\gldim\hh{Q}-\epsilon$.
Similarly we have the inequality in the other direction, which completes the claim (and hence the proof).
\end{proof}

%=========================================================
\subsection{The inducing theorem}
%=========================================================

\begin{lemma}\label{lem:Hom}
Suppose $\sli$ is either $\padd$ in \eqref{eq:padd} or $\pext$ in \eqref{eq:pext},
where $\np$ is the slicing of a stability condition $\D_\infty$ of an $\XX$-baric heart of $\ns$
of some Calabi-Yau-$\XX$ category $\D_\XX$.
If $\Re(s)\geq\gldim\ns$, then
$$\Hom_{\D_\XX}(\sli(\phi_1),\sli(\phi_2))=0$$ for $\phi_1>\phi_2$.
\end{lemma}
\begin{proof}
We need to show that for any $\phi_1>\phi_2, k_1,k_2\in\ZZ$,
\begin{gather}\label{eq:v}
    \Hom_{\D_\XX}(\np(\phi_1-k_1\Re(s))[k_1\XX],\np(\phi_2-k_2\Re(s))[k_2\XX])=0.
\end{gather}
There are three cases:
\begin{itemize}
  \item If $k_2-k_1<0$ or $k_2-k_1>1$, \eqref{eq:v} follows from \eqref{eq:01}.
  \item If $k_1=k_2$, the left hand side of \eqref{eq:v}
  equals $$\Hom_{\D_\infty}(\np(\phi_1-k_1\Re(s)),\np(\phi_2-k_1\Re(s))),$$
  which is zero since $\phi_1-k_1\Re(s)>\phi_2-k_1\Re(s)$.
  \item If $k_2-k_1=1$, then apply Calabi-Yau-$\XX$ duality, the left hand side of \eqref{eq:v} equals
  $$D\Hom_{\D_\infty}(\np(\phi_2-k_2\Re(s)),\np(\phi_1-k_1\Re(s))).$$
  Since
  \[\begin{array}{rl}
    &\left( \phi_1-k_1\Re(s) \right) - \left( \phi_2-k_2\Re(s) \right)\\
  =&\phi_1-\phi_2+\Re(s)\\
  >&\gldim\np,
  \end{array}\]
  the $\Hom$ vanishes.
\end{itemize}
\end{proof}

\begin{theorem}\label{thm:inducing}
Let $\D_\XX$ be a Calabi-Yau-$\XX$ category satisfies Assumption~\ref{assumption:R}.
Given a stability condition $\ns=(\nz,\np)$ on an $\XX$-baric heart $\D_\infty$ of $\D_\XX$,
then we have the following.
\begin{enumerate}
\item The induced additive pre-stability condition $\sadd=(Z, \padd)$ is
a stability condition on $\D_\XX$
if and only if
\begin{gather}\label{eq:cond}
    \Hom(\np(\phi_1),\np(\phi_2))=0 \quad \text{for any $\;\phi_2-\phi_1\geq\Re(s)-1$.}
\end{gather}
\item The induced extension pre-stability condition $\sext=(Z, \pext)$ is
a stability condition on $\D_\XX$
if and only if
\begin{gather}\label{eq:closed}
    \gldim\ns\le\Re(s)-1.
\end{gather}
\end{enumerate}
Clearly, both $\sadd$ and $\sext$ satisfies \eqref{eq:X=s} and hence
$(\sadd,s)$ and $(\sext,s)$ are both $q$-stability conditions.
Finally, they satisfy $q$-support property, where $\Grot(\D_\infty)$ provides the $\ZZ^n$ lattice in Condition~$1^\circ$
of Definition~\ref{def:support}.
\end{theorem}
\begin{proof}
We only prove the result for additive case where the extension case is just a slight variation.
Note that $q$-support property can be checked directly once we have shown
$(\sadd,s)$ and $(\sext,s)$ are both $q$-stability conditions.

First, we prove the `if' part.
We need to show that $\sigma=(q_s\circ Z_q, \sli)$ is a stability condition on $\D_\XX$,
where $\sli$ is defined as $\sli=\add^s\np[\ZZ\XX]$.
Clearly, $Z=q_s\circ Z_q$ is a group homomorphism that is compatible with $\sli$
in the sense that
$$ Z(E)=m(E) e^{\mathbf{i} \pi \phi} $$
for some $m(E)\in\mathbb{R}_{>0}$ if $E\in\sli(\phi)$.
By Lemma~\ref{lem:Hom},
$\sli$ satisfies the $\Hom$-vanishing properties.
Thus what is left to show is any object exists (and hence unique)
a HN-filtration (with respect to $\sli)$.

Any object $M$ in $\D_\XX$ admits a filtration, with respect to the $\XX$-baric heart $\D_\infty$,
\begin{gather}\label{eq:filtx}
    \filt_{\XX}(M)=\{E_1[k_1\XX],\ldots,E_m[k_m\XX]\mid E_i\in\D_\infty, k_1>\cdots>k_m (\in\ZZ) \}.
\end{gather}
Moreover, each $E_i\in\D_\infty$ admits a filtration with respect to the slicing $\np$
and hence \eqref{eq:filtx} can be refined as
\begin{gather}\label{eq:filt}
\begin{array}{rll}
    \filt_0(M)=\{&E_{1,1}[k_1\XX],\ldots,&E_{1,l_1}[k_1\XX],\\
    &E_{2,1}[k_2\XX],\ldots,&\\
    &\qquad\qquad\ldots\ldots&E_{m,l_m}[k_m\XX]\mid E_{i,r}\in\np(\phi_{i,r})\\
    &k_1>\cdots>k_m,& l_1,\ldots,l_m\in\ZZ \}
\end{array}
\end{gather}
or simply as
\begin{gather}\label{eq:filt0}
    \filt_0(M)=\{ F_1,\ldots,F_t\mid F_j\in\sli(\phi_j)\}.
\end{gather}
Now we claim that we can exchange the position of the factors in $\filt_0(M)$ inductively
so that it becomes a filtration for $\sli$
\[
    \filt(M)=\{ F_1',\ldots,F_t'\mid F_j'\in\sli(\phi_j'),
    \phi_1'>\cdots>\phi_t'\}.
\]
This is equivalent to show that (and then use induction)
\begin{itemize}
\item for $j$ in \eqref{eq:filt0} with $\phi_j<\phi_{j+1}$,
we have
\begin{gather}\label{lem:FF}
    \Hom_{\D_\XX}(F_{j+1},F_j[1])=0,
\end{gather}
which implies, by the Octahedral Axiom, that the exchange of $F_j$ and $F_{j+1}$ is admissible.
\end{itemize}
Note that $F_j=M[a\XX], F_{j+1}=L[b\XX]$ for some $M,L\in\D_\infty, a\geq b\in\ZZ$.
There are three cases:
\begin{itemize}
\item If $a=b$, then $F_j,F_{j+1}$ are factors of the same $E_i$ (with respect to the slicing $\np$)
and hence $\phi_j>\phi_{j+1}$ which contradicts to $\phi_j<\phi_{j+1}$.
\item If $a>b+1$, then \eqref{lem:FF} follows from \eqref{eq:01}.
\item If $a=b+1$,
then $\phi_j<\phi_{j+1}$ is equivalent to
\[ \varphi_{\np} (M)+\Re(s) < \varphi_{\np} (L) ,\]
where $\varphi_{\np}$ denotes the phase of objects in $\D_\infty$ with respect to $\np$.
Hence
\[ \gldim\np<\Re(s)-1 = \varphi_{\np} (L)-\varphi_{\np} (M[1]), \]
which implies $\Hom_{\D_\infty}(M[1],L)=0$.
Therefore we have
\[
    \Hom_{\D_\XX}(F_{j+1},F_j[1])=\Hom_{\D_\XX}(L,M[1+\XX])=D\Hom_{\D_\XX}(M[1],L)=0.
\]
\end{itemize}
In all, $\sli$ is a slicing, $(q_s\circ Z_q, \np)$ is in $\Stab\D_\XX$
and
$\sigma=\hh{L}_*^s(\ns)=(Z_q,\np,s)$ is in $\OStab_s\D_\XX$.

Next, let us prove the `only if' part.
Suppose that \eqref{eq:cond} does not hold, then there exists
$M,L$ in $\D_\infty$ such that
$$\Hom_{\D_\infty}(M[1],L)\neq0\quad
\text{and}\quad \varphi_{\np}(L)-\varphi_{\np}(M[1])\geq\Re(s)-1,$$
where $\varphi_{\np}$ is the phase with respect to the slicing $\np$.

Now consider $M,L$ in $\D_\XX$. We have
\begin{gather}\label{eq:non}
    0\neq D\Hom_{\D_\XX}(M[1],L)=\Hom_{\D_\XX}(L,M[1+\XX])
\end{gather}
and (here $\varphi_{\np}$ denote the phase with respect to $\np$)
\begin{gather}\label{eq:contra}
    \varphi_{\np}(L)-\varphi_{\np}(M[\XX])=
    \varphi_{\np}(L)-\varphi_{\np}(M)-\Re(s)\geq0.
\end{gather}
By \eqref{eq:non}, there exists an object $E$ sits in the nontrivial triangle
$$M[\XX]\to E\to L\to M[1+\XX].$$
Consider the filtration of $E$ with respect to the $\XX$-baric heart $\D_\infty$,
which must be the triangle above, i.e.
\begin{gather}\label{eq:infiltatioin}
    \filt_{\XX}(E)=\{M[\XX],L\}.
\end{gather}
%By definition, any semi-stable objects will be in $\XX$-shifts $\np$;
%thus $E$ is not semi-stable.
Consider the filtration of $E$ with respect to the slicing $\sli=\add^s\np[\ZZ\XX]$ %(with $t\geq2$):
\begin{gather}\label{eq:infiltatioin0}
    \filt(E)=\{ F_1,\ldots,F_t\mid F_j\in\sli(\phi_j),
    \phi_1>\cdots>\phi_t\}.
\end{gather}
Again, using the Octahedral Axiom (and vanishing $\Hom$) we can rearrange
the order of the factors so that it becomes
\[
    \filt(E)=\{ F_1',\ldots,F_t'\mid F_j'=E_i[k_i\XX], E_i\in\hh{R}_0,
    k_1\geq\cdots\geq k_t\}.
\]
Comparing with the filtration \eqref{eq:infiltatioin},
we deduce that $1=k_1=\cdots=k_j,k_{j+1}=\cdots=k_t=0$ and
there are filtrations
\begin{gather*}
    \filt(M)=\{ E_1,\ldots,E_j \},\\
    \filt(L)=\{ E_{j+1},\ldots,E_t \}.
\end{gather*}
As $M$ and $L$ are (semi)stable objects, then the phase of $E_i$ equal phase of $M$
(with respect to $\sli$) for $1\leq i\leq j$ and
the phase of $E_i$ equal phase of $L$ (with respect to $\sli$) for $j+1\leq i\leq t$.
As $\{E_i[k_i]\}$ are rearrangement of $F_i\in\sli(\phi_j)$, we deduce
that $j=1$ and $t=2$, i.e. the filtration \eqref{eq:infiltatioin}
coincide with the filtration \eqref{eq:infiltatioin0}.
However, then \eqref{eq:contra} implies $\phi_1=\varphi_{\np}(M[\XX])\leq\varphi_{\np}(L)=\phi_2$
that contradicts to $\phi_1>\phi_2$, which completes the proof.
\end{proof}

\begin{definition}\label{eq:def:qstab}
An open (induced) $q$-stability condition on $\D_{\XX}$ is a pair $(\sigma,s)$
consisting of a stability condition $\sigma$ on $\D_\XX$ and a complex parameter $s$,
satisfying
\begin{itemize}
    \item $\sigma=\sadd$ is an additive pre-stability condition induced from some triple
    $(\D_\infty,\ns,s)$
    as in Construction~\ref{con:q} with
\begin{equation}\label{eq:open}
    \gldim\ns+1<\Re(s)
\end{equation}
\end{itemize}
Denote by $\OStab_s\D_\XX$ the set of all open $q$-stability conditions with the parameter $s\in\CC$
and by $\OStab\D_\XX$ the union of all $\OStab_s\D_\XX$.

Similarly, a closed (induced) $q$-stability condition on $\D_{\XX}$ is a pair $(\sigma,s)$
consisting of a stability condition $\sigma$ on $\D_\XX$ and a complex parameter $s$,
satisfying
\begin{itemize}
    \item $\sigma=\sext$ is an extension pre-stability condition induced from some triple
    $(\D_\infty,\ns,s)$. Note that Theorem~\ref{thm:inducing} forces the inequality \eqref{eq:closed} to hold.
\end{itemize}
Denote by $\CStab_s\D_\XX$ the set of all closed $q$-stability conditions with the parameter $s\in\CC$.
\end{definition}

By comparing inequalities \eqref{eq:open} and \eqref{eq:closed}
one deduces that $\OStab_s\D_\XX\subset\CStab_s\D_\XX$.
We prefer to consider open $q$-stability conditions as we can glue them together.
In fact, in most of the cases we are interested in
(with $\Re(s)\ge2$), we expect they coincide.

\begin{theorem}\label{thm:manifold}
Let $\D_\XX$ be a Calabi-Yau-$\XX$ category satisfying Assumption~\ref{assumption:R}.
Then $\OStab\D_\XX$ is a complex manifold of dimension $n+1$.
\end{theorem}
\begin{proof}
Given an open $q$-stability condition $(\sigma,s)$, suppose that
$\sigma$ is induced from a triple $(\D_\infty,\ns,s)$.
Then by Theorem~\ref{thm:inducing}, we have \eqref{eq:cond}.
As $\gldim$ is continuous by Lemma~\ref{lem:conti},
there is a neighbourhood $U(\ns)$ of $\ns$ in $\Stab\D_\infty$ satisfying
\[ \Re(s)>\gldim\ns'+1+\epsilon\]condition
for any $\ns'\in U(\ns)$ and some positive real number $\epsilon$.
Hence, for any $\ns'$ in $U(\ns)$ and $s'\in(s-\epsilon,s+\epsilon)$,
the triple $(\D_\infty,\ns',s')$ induces an open $q$-stability.
This gives a local chart for $\OStab\D_\XX$, that isomorphic to
\[
    U(\ns)\times(s-\epsilon,s+\epsilon).
\]
This type of charts provide the required complex manifold structure.
\end{proof}

\begin{remark}
One of the key feature of the local charts we construct in Theorem~\ref{thm:manifold}
is that each point $(\sigma,s)$ admits a distinguished section $\Gamma=\{(\sigma',s')\}$
for $s'\in(s-\epsilon,s+\epsilon)$, satisfying
\begin{itemize}
  \item the sets of (semi)-stable objects are invariants along this section.
\end{itemize}
\end{remark}

By Theorem~\ref{thm:inducing}, the minimal value of the global function
on the $\XX$-baric heart is important concerning the question that
if $\QStab_s\D_\XX$ is empty or not.
%We have the following conjecture.
%\begin{conjecture}
%If $\Re(s)>\inf\gldim\Stab\D_\infty+1$ for some $\XX$-baric heart of $\D_\XX$, then
%\begin{gather}\label{eq:Q=X}
%    \OStab_s\D_\XX=\CStab_s\D_\XX
%\end{gather}
%and they consist of connected components of $\QStab_s\D_\XX$.
%\end{conjecture}

%In fact, we expect $\QStab_s\D_\XX$ is connected (when $\Re(s)\geq2$).
%When $\Re(s)=2$, we expect $\OStab_s\D_\XX$ is a open dense subspace of $\QStap_s\D_\XX$.
%When $\Re(s)<2$, $\QStab_s\D_\XX$ may not be connected
%and possibly \eqref{eq:Q=X} holds again.
%However, we have not look into such cases at the moment.
%\subsection{Comparison with King-Qiu's inducing of stability conditions}
%The rest of the paper will explore our motivating examples for Calabi-Yau-$\XX$ categories
%that admits interesting $q$-stability conditions
%and discuss various applications.
\subsection{Example of acyclic quivers}
Here is a class of examples of $\XX$-baric hearts.
\begin{proposition}\label{prop:X-heart}
$\DQ$ is an $\XX$-baric heart of $\DXQ$.
\end{proposition}
\begin{proof}
Regard $\Gamma_{\XX}Q$ as a $\ZZ\XX$-graded algebra.
As $\XX$ is the grading shift,
any object admits an (HN-)filtration \eqref{eq:X-HN}.
The Hom vanishing property comes from the fact that
$\Gamma_{\XX}Q$ is concentrated in non-positive degrees (with respect to the $\ZZ\XX$ grading).
Therefore, $\DQ$ is an $\XX$-baric heart of $\DXQ$.
\end{proof}
Thus, we can apply our general results, namely
Theorem~\ref{thm:inducing} and Theorem~\ref{thm:manifold}.
In particular, we can construct $q$-stability conditions on $\DXQ$
for any $s$ with $\Re(s)\ge2$.

%=========================================================
\section{Perfect derived categories as cluster-$\XX$ categories}\label{sec:cluster}
%=========================================================
Let $Q$ be an acyclic quiver as above.
%=========================================================
\subsection{Cluster categories}
%=========================================================
The cluster categories $\C(Q)$ were introduced in \cite{BMRRT}
to categorify cluster algebras associated to $Q$.
Keller \cite{K2} provided the construction of cluster categories as orbit categories.

\begin{definition}\cite{BMRRT, K2}\label{def:cluster}
For any integer $m\geq 2$, the \emph{$m$-cluster shift} is the auto-equivalence of $\per\k Q\cong\DQ$ given by
$\shift{m}=\tau^{-1}\circ[m-1]$.
The \emph{$m$-cluster category} $\C_m(Q)$ is the orbit category
\begin{gather}\label{eq:clusterm}
    \C_m(Q)\colon=\DQ/\shift{m}
\end{gather}
\end{definition}

Note that $\C_m(Q)$ is Calabi-Yau-$m$ and
the classical case (corresponding to cluster algebras) is when $m=2$.

Another way to realize cluster categories is via Verdier quotient.
Set $$N=m+1$$ and let
\begin{gather}\label{eq:VV}
    \C(\qq{N})\colon=\per\qq{N}/\DNQ
\end{gather}
be the generalized cluster category associated to $\qq{N}$,
i.e. it sits in the short exact sequence of triangulated categories:
\begin{gather}\label{eq:A}\xymatrix{
    0 \ar[r] & \DNQ \ar[r]& \per\qq{N} \ar[r] & \C(\qq{N}) \ar[r] &0.
}\end{gather}
Note that $\C(\qq{N})$ is Calabi-Yau-$(N-1)$.

\begin{theorem}\cite{A,Guo,Kel1}\label{thm:AGK}
There is a natural triangle equivalence $\C_{m}(Q)\cong\C(\qq{N})$
that identifies the canonical cluster tilting objects in them.
\end{theorem}

Here, a cluster tilting object in a Calabi-Yau-$m$ cluster category
is the direct sum of a maximal collection of non-isomorphic indecomposables $\{M_i\}$
such that $\Ext^k(M_i, M_j)=0$, for all $1\leq k\leq m-1$.
%=========================================================
\subsection{$N$-reduction}
%=========================================================
In Definition~\ref{def:CYXQ}, replacing $\XX$ with an integer $N\ge2$,
we obtain the usual Ginzburg dg algebra $\qq{N}$
and the corresponding Calabi-Yau-$N$ category $\DNQ$.
On the level of differential (double) graded algebras,
there is a projection
\begin{equation}\label{eq:projection}
    \pi_N\colon\qq{\XX}\to\qq{N}
\end{equation}
collapsing the double degree $(a,b)\in\ZZ\oplus\ZZ\XX$ into $a+bN\in\ZZ$,
that induces a functor
\[
    \pi_N\colon\DXQ\to\DNQ.
\]
Thus, we obtain the following
(cf. \cite[Prop~4.18]{ST}, \cite{K2}, \cite[Theorem~5.1]{KaY}).

\begin{proposition}\label{pp:ST}
$\DXQ \slash [\XX-N]$ is $N$-reductive, where the triangulated structure is provided
by its unique triangulated hull $\DNQ$.
\end{proposition}
%\begin{proof}
%By Kozsul duality, we have $\DXQ\cong\per\ee_Q^\XX,$
%where $$\ee_Q^\XX=\RHom^{\ZZ^2}_{\DXQ}(S_Q^\XX, S_Q^\XX)$$
%for $S_Q^\XX$ is the direct sum of simple $\qq{\XX}$-modules $S_i^\XX$.
%Similarly, we have Calabi-Yau-$N$ version $\DNQ\cong\per\ee_Q^N$
%where $$\ee_Q^N=\RHom^\bullet_{\DNQ}(S_Q^N, S_Q^N)$$
%for $S_Q^N$ is the direct sum of simple $\qq{N}$-modules $S_i^N$.

%Then $\pi_N$ maps the shifts of projective $S_i^\XX[a+b\XX]$ of $\per\ee_Q^\XX$
%to the shifts of projectives $S_i^N[a+bN]$ of $\per\ee_Q^N$,
%which extends to other objects via their projective resolutions.
%\end{proof}

Thus, in this case we have
$$\DNQ=\DXQ \sslash [\XX-N].$$
Similarly, there is an $N$-reduction
$\pi_N\colon\per\qq{\XX}\to\per\qq{N}$.
A direct corollary is the following.

\begin{corollary}\label{cor:XN}
$\pi_N$ induces an injection $i_N\colon\Aut\DXQ/[\XX-N]\to\Aut\DNQ$.
\end{corollary}
%\begin{proof}
%This follows from the proposition above,
%that the kernel of $$\Aut\DXQ\to\Aut\DNQ$ is exactly $\ZZ[\XX-N].$$
%\end{proof}

%=========================================================
%\subsection{Braid group action in the Dynkin case}
%=========================================================
For $\DNQ$, we also have the corresponding spherical twist group $\ST_N(Q)$,
generated by spherical twists along the $N$-spherical simple $\qq{N}$-modules.
When $Q$ is of type $A$, \cite{KhS} shows $\iota_Q^\XX$ in \eqref{eq:FF} is an isomorphism.
\cite{ST} proves the corresponding
$$\iota_Q^N\colon\Br_Q \to \ST_N(Q)$$
is also an isomorphism using Proposition~\ref{pp:ST}.
Recall the following result for the Dynkin case and affine type A case.

\begin{theorem}
$\iota_Q^N$ is an isomorphism when
\begin{itemize}
\item $Q$ is a Dynkin quiver and $N\ge2$ (\cite[Thm.~B]{QW}).
\item $Q$ is affine type A and $N=2$ (\cite[Cor.~37]{IUU}) or $3$ (\cite[Thm.~7.3]{QQ}).
\end{itemize}
\end{theorem}

Then combining with Corollary~\ref{cor:XN}, we obtain Calabi-Yau-$\XX$ version
as a direct corollary of Corollary~\ref{cor:XN}.
\begin{theorem}\label{thm:QW}
Let $Q$ either be a Dynkin quiver or an affine type A quiver,
Then $\iota_Q^\XX:\Br_Q \to \ST_\XX(Q)$ is an isomorphism for any $N\ge2$.
\end{theorem}
%=========================================================
\subsection{Cluster-$\XX$ categories}
%=========================================================
As a generalization,
we defined the cluster-$\XX$ category $\C(\qq{\XX})$ as the Verdier quotient $\per\qq{\XX}/\DXQ$.

\begin{theorem}\label{thm:p=c}
The embedding $\k Q\to\qq{\XX}$ induces a triangle equivalence
\[\C(\qq{\XX})\xrightarrow{\cong}\per\k Q\]
that factors through $\per\QQX$.
\end{theorem}
\begin{proof}
First of all, the embedding $\k Q\to\qq{\XX}$ induces the functor
\begin{gather}\label{eq:induce p}
  i_*\colon  \per\k Q\to \per\qq{\XX}
\end{gather}
sending projectives to projectives.
Denote the image of $i_*$ by $\hh{C}$, which is generated by $\{ \qq{\XX}[i]\}_{i\in\ZZ}$.
%Let
%\[
%    \hh{M}=\add\bigoplus_{j\in\ZZ}\qq{\XX}[j\XX],
%\]
%which is a silting subcategory of $\per\qq{\XX}$
%that induces a torsion pair $\<\widetilde{\hh{X}},\widetilde{\hh{Y}}\>$ of $\per\qq{\XX}$ for
%$\widetilde{\hh{X}}=(\hh{M}[\ZZ_{\le0}])^\perp$ and $\widetilde{\hh{Y}}=(\hh{M}[\ZZ_{>0}])^\perp$.
%Here, the notation $?[\ZZ_{\le0}]$ means $\bigcup_{j\in \ZZ_{\le0}}?[j]$
%and similar for $?[\ZZ_?\XX]$.

Recall that $\DXQ$ admits the Serre functor $\XX$ and $\Sim\qq{\XX}$ is the set of simple $\qq{\XX}$-modules.
Denote by $\hh{S}$ the thick subcategory of $\DXQ$ generated by the objects in $\Sim\qq{\XX}[\ZZ].$
Consider the canonical unbounded t-structure
$\DXQ=\<\hh{X}, \hh{Y}\>$,
where $\hh{X}$ is generated by $\hh{S}[\ZZ_{\ge0}\XX]$ and
$\hh{Y}$ is generated by $\hh{S}[\ZZ_{<0}\XX]$.

By the Calabi-Yau-$\XX$ duality in Lemma~\ref{lem:cy},
we obtain
\[
    \Hom(M,\QQX)\cong \Hom(\QQX,M[\XX])^*
\]
for any $M\in\DXQ$.
Noticing the following calculation about $\Hom$ between projectives and simples:
\begin{gather}\label{eq:2}
    \Hom^{\ZZ}_{\per\QQX}(\QQX,\hh{S}[j\XX])=0
\end{gather}
for any $j\neq0$,
we have
\begin{gather}\label{eq:1}
    \Hom_{\per\QQX}(\hh{S}[j\XX],\hh{C})=0
\end{gather}
for any $j\neq-1$ and, moreover,
\[
    \Hom_{\per\QQX}(\hh{S},\QQX[\XX])
    \cong\Hom_{\per\QQX}(\QQX,\hh{S})^*\neq0,
\]
which implies $\Hom_{\per\QQX}(\hh{S},\hh{C}[\XX])\neq0$.
Therefore we deduce that the right perpendicular
\[
    \hh{X}^\perp%=\hh{X}^\perp\mid_{\per\QQX}
        \colon=\{ Z\in\per\QQX \mid \Hom_{\per\QQX} (\hh{X},Z)=0 \}
\]
of $\hh{X}$ in $\per\QQX$ is generated by
$\hh{C}[\ZZ_{\leq0}\XX]$ and, similarly,
the left perpendicular
\[
    ^\perp\hh{Y}\colon=\{ Z\in\per\QQX \mid \Hom_{\per\QQX} (Z,\hh{Y})=0 \}
\]
of $\hh{Y}$ in $\per\QQX$ is generated by $\hh{C}[\ZZ_{\geq0}\XX]$.

Next we claim that
$\<\hh{X},\hh{X}^\perp\>$ and $\<^\perp\hh{Y},\hh{Y}\>$ are torsion pairs in $\per\qq{X}$.
We only need to show that they generate $\per\qq{X}$.
Recall that the notation $\<\hh{A},\hh{B}\>$ consists of object $M$ that
admits a triangle $$A\to M\to B\to A[1]$$
with $A\in\hh{A}$ and $B\in\hh{B}$.

Consider the $\hh{Y}$ case first.
We start to show that $\hh{C}\subset\<^\perp\hh{Y}[\XX], \hh{S}\>$.
Take any projective $P_i\in\hh{C}$ for $i\in Q_0$.
Then there is a triangle (cf. \cite[\S~2.14]{KY})
\begin{gather}\label{eq:PtoS}
    S_i[-1]\to A_1\oplus B_1
        \to P_i \to S_i,
\end{gather}
for
\[A_1=\bigoplus_{ i\xrightarrow{a}j \in Q_1 } a\ P_j
    \quad\text{and}\quad
  B_1=\bigoplus_{ i\xrightarrow{b}k \in Q_1^*\cup Q_0^* } b\ P_k.
\]
The arrows $b\in Q_1^*\cup Q_0^*$ are the new arrows in the double $\overline{Q}$ with $\XX$-degree $-1$
(see Definition~\eqref{def:CYXQ}). Thus, $b P_k$'s are in $^\perp\hh{Y}[\XX]$ and hence $B_1\in{^\perp}\hh{Y}[\XX]$.
As for $a\in Q_1$, their $\XX$-degree is 0. Thus $P_j$'s are in still ${^\perp}\hh{Y}$ and hence $A_1\in{^\perp}\hh{Y}$.
Using \eqref{eq:PtoS} for $P_j\to S_j$ and we will get the following filtration of $P_i$
\[
\xymatrix @C=4mm{
 A_2\oplus B_2 \ar[rr]   &&  A_1\oplus B_1 \ar[dl] \ar[rr] && P_i, \ar[dl]\\
& \displaystyle\bigoplus_{ i\xrightarrow{a}j \in Q_1 } a\ S_j \ar@{-->}[ul] && \quad S_i\quad \ar@{-->}[ul]
}\]
for
$$ A_2= \bigoplus_{ \substack{i\xrightarrow{a}j\in Q_1 \\
        j\xrightarrow{a'}j' \in Q_1}} a'a\ P_{j'}
        \quad\text{and}\quad
        B_2=B_1\oplus \bigoplus_{ \substack{i\xrightarrow{a}j\in Q_1\\
        j\xrightarrow{b'}k' \in Q_1^*\cup Q_0^*}} b'a\ P_{k'} .$$
As before, we have $B_2\in{^\perp}\hh{Y}[\XX]$ and $A_2\in{^\perp}\hh{Y}$.
Again, we will further decompose $P_{j'}$'s using the corresponding \eqref{eq:PtoS}.
Since $\k Q$ is finite dimensional, this process will end up with a filtration
\begin{equation}\label{eq:HNx}
\xymatrix @C=5mm{
 B_{l+1} \ar[rr]   &&  A_{l}\oplus B_{l} \ar[dl] \ar[rr] &&  \ar[dl]
 \ar[r] & \dots  \ar[r] & A_1\oplus B_1 \ar[rr] && P_i \ar[dl] \\
& H_{l+1} \ar@{-->}[ul] && H_l \ar@{-->}[ul] &&&& H_1 \ar@{-->}[ul]
}
\end{equation}
for $H_m$ is a direct sum whose summands are of the form $p S_{t{p}}$,
where $p$ is a non-zero path of length less than $m$ in $Q_1$ with $h(p)=i$, and
$B_m$ is a direct sum whose summands are of the form $b p P_{t(b)}$,
where $p$ is a non-zero path of length less than $m$ in $Q_1$ and $b\in Q_1^*\cup Q_0^*$
with $h(b)=t(p), h(p)=i$.
Here $h(-)$ and $t(-)$ are the head/tail function for arrows/paths.
As above, we have $B_m\in{^\perp}\hh{Y}[\XX]$.
Also, we have $H_m\in\hh{S}$.
Therefore \eqref{eq:HNx} implies that $P_i\in\<{^\perp}\hh{Y}[\XX], \hh{S}\>$ via Octahedron Axiom.
Thus we obtain $\hh{C}\subset{^\perp}\<\hh{Y}[\XX], \hh{S}\>$.
Inductively, with $\hh{C}[-m\XX]\subset\<\hh{Y}[-m\XX+\XX], \hh{S}[-m\XX]\>$ holds for any
positive integer $m$, we deduce that
$\hh{C}[-m\XX]\subset\<{^\perp}\hh{Y}, \hh{Y}\>$ holds for $m>0$ since
$${^\perp}\hh{Y}[-m\XX-\XX]\subset{^\perp}\hh{Y}\quad\text{and}\quad\hh{S}[-m\XX]\subset\hh{Y}.$$
Hence $\per\qq{X}=\<{^\perp}\hh{Y},\hh{Y}\>$ as required.

For the $\hh{X}$ case, we only need to modify \eqref{eq:PtoS} as follows.
Rewrite $$B_1=e_i^*\ P_i \bigoplus_{ i\xrightarrow{b}j \in Q_1^*} b\ P_k=P_i[\XX-1]\oplus \underline{B_1}.$$
and by Octahedron Axiom we have
\[
  \xymatrix@C=3pc{
    & (A_1\oplus\underline{B_1})[-1] \ar[d]\ar@{=}[r] & (A_1\oplus\underline{B_1})[-1] \ar[d]\\
    S_i[-1] \ar@{=}[d] \ar[r] & P_i[\XX-1] \ar[r] \ar[d]
        & M \ar[d] \ar[r] & S_i \ar@{=}[d]\\
    S_i[-1] \ar[r] & A_1\oplus B_1 \ar[d] \ar[r] &
        P_i \ar[d]\ar[r] & S_i\\
    & A_1\oplus\underline{B_1} \ar@{=}[r]& A_1\oplus\underline{B_1}\\
  }
%\end{equation}
\]
So we can decompose $P_i$ into a filtration with factors $S_i[-\XX]$, $A_1\oplus \underline{B_1}$ and $P_i[1-\XX]$,
similar to the first step of \eqref{eq:HNx}. In the same way, we can prove that
$\<\hh{X},\hh{X}^\perp\>=\per\qq{X}$.

Finally, by \cite[Thm~1.1]{IY}, we obtain the following
equivalence $F$ between additive categories
\begin{gather}\label{eq:IY}
\xymatrix{
    F\colon
    \hh{C}\quad\ar[rr]^{\cong}\ar@{^{(}->}[dr]&&\hh{C}(\QQX)\\
    &\per\QQX\ar@{->>}[ur]^{\pi},
}\end{gather}
for $\hh{C}=\hh{X}{^\perp}\cap\;{^\perp}\hh{Y}[1]$.
What is left to show is that $F\circ i_*$ is an equivalence as triangulated categories.
This follows from the fact that $i_*$ is a functor between triangulated categories
and $F$ preserves shifts and triangles.
\end{proof}

%Note that in the acyclic quiver case, we have $\per\k Q\cong\DQ$.
Combining results above, we obtain the following.

\begin{corollary}\label{cor:SES}
We have the following commutative diagram between short exact sequences of
triangulated categories:
\begin{equation}
\xymatrix{
    0 \ar[r] &
        \DXQ \ar[r]\ar[d]^{_{\sslash[\XX-N]}} &
        \per\qq{\XX} \ar[r]\ar[d]^{_{\sslash[\XX-N]}} &
        \per\k Q \ar[r]\ar[d]^{_{/\tau[2-N]}} &0 \\
    0 \ar[r] & \DNQ \ar[r]& \per\qq{N} \ar[r] & \C_{N-1}(Q) \ar[r] &0
}.\end{equation}
\end{corollary}
\begin{proof}
The first two vertical (exact) functors are induced from projection \eqref{eq:projection}
as in Proposition~\ref{pp:ST}
and the left square commutes as the corresponding horizonal functors are just inclusions.
The third vertical functor is Keller's orbit quotient in Definition~\ref{def:cluster}.
The right square commutes since the projectives (generators) of $\per\qq{\XX}$
are mapped to projectives of $\per\k Q$ and $\per\qq{N}$ respectively,
and they further become the canonical $(N-1)$-cluster tilting object in $\C_{N-1}(Q)$.
\end{proof}

%=========================================================
\section{Example: Calabi-Yau-$\XX$ $A_2$ quivers}\label{sec:A2}
%=========================================================
In this section, we discuss the example of $\OStab\D_\XX(Q)$ for $Q$ is an $A_2$ quiver.
%We require that $s\in\RR$ but everything should work for $s\in\CC$ (possibly with some variation).
The prototype, i.e. the $N$-fiber
$$\OStab_N\D_\XX(A_2)\cong\Stab\D_N(A_2)$$ for $N=2$,
is calculated in \cite{BQS}.
\begin{figure}[ht]\centering
\begin{tikzpicture}[scale=.35]
\path (0,0) coordinate (O);
\path (4,0) coordinate (A);
\path (3+0.2,10) coordinate (m1);
\path (3+0.2,-10) coordinate (m2);
\draw[fill=gray!14] (m1)
    .. controls +(-90:3) and +(120:1) .. (A)
    .. controls +(-120:1) and +(90:3).. (m2)
   to [out=180,in=0] (-15,-10)
   .. controls +(90:3) and +(-90:1) ..  (-15,10) ;
\draw[white,thick] (m2) edge (-15,-10);
\draw[dotted,->,>=stealth] (-20,0) -- (8,0) node[right]{$x$};
\draw[dotted,->,>=stealth] (0,-10) -- (0,10) node[above]{$y$};
\draw (0,0) node[below left]{$0$};
\path[dotted] (6,-10) edge (6,10);
\path[dotted] (3,-10) edge (3,10);
\path[thick] (-15,0) edge (-15,10);
\path[thick] (-15,-10) edge (-15,0);
\draw[red,thick] (A) edge (6,0);
\draw[red,thick] (m1) .. controls +(-90:3) and +(120:1) .. (A);
\draw[red,thick] (A)  .. controls +(-120:1) and +(90:3).. (m2);
\draw (6,0) node[below right] {$1$};
\path (5.5,0) node (a) {$$};
\path (3.6,1.5) node (b) {$$};
\path (3.6,-1.5) node (c) {$$};
\path[->,>=stealth,orange, bend right] (a) edge (b);
\path[->,>=stealth,orange, bend right] (b) edge (c);
\path[->,>=stealth,orange, bend right] (c) edge (a);
\path (-15,1+4) node (d) {$ $};
\path (-15,-1+4) node (e) {$ $};
\path[orange, bend right=45] (d) edge (-16,0+4);
\path[->,>=stealth,orange, bend right=45] (-16,0+4) edge (e);
\path[orange, bend right=45] (e) edge (-15+1,0+4);
\path[->,>=stealth,orange, bend right=45] (-15+1,0+4) edge (d);
\path (4,6) node {$l_+$};\path (4,-6) node {$l_-$};
\draw[fill=black] (0,0) circle (.05);
\draw[fill=white] (6,0) circle (.1);
\draw[fill=red] (4,0) circle (.2);
\draw[fill=black] (-15,0+4) circle (.2);
%
%\path (-6,-10.6) node[gray]
%{$ \underbrace{\qquad\qquad\qquad\qquad\qquad\qquad\qquad\qquad\qquad} $};
%\path (-6,-11.5) node {$W_n$};
%
\path (4,0) node[below right] {\tiny{$\frac{2}{3}$}};
\path (-15,0) node{$\cdot$} node[left] {\tiny{$\frac{2-\Re(s)}{2}$}};
\path (-15,0+4) node[right] {{$\frac{2-s}{2}$}};
\end{tikzpicture}
\caption{The region/fundamental domain $R_s$}\label{fig:A2N}
\end{figure}

Let $S_1$ and $S_2$ be the simple $\Gamma_\XX A_2$-module satisfying
\[
    \Hom^{\ZZ^2}(S_1,S_2)=\k[-1],\quad
    \Hom^{\ZZ^2}(S_2,S_1)=\k[-\XX+1]
\]
and $\twi_i$ be the corresponding spherical twists.
The canonical $\XX$-baric heart $\D_\infty(A_2)$ is generated by the shifts of $S_1$ and $S_2$
and has one more indecomposable object $\twi_1(S_2)$ (up to shift).
Moreover, $\D_\infty(A_2)$ admits a (normal) heart generated by $S_1$ and $S_2$.
By Corollary~5.2 and Theorem~5.4, we have the following.
\begin{itemize}
  \item $\ST_\XX A_2=\<\twi_1,\twi_2\>\cong\Br_3$.
  \item The center of $\ST_\XX A_2$ is generated by $(\twi_1\circ\twi_2)^3$.
  \item Let $\tau_\XX=\twi_1\circ\twi_2\circ[\XX-2]$ which satisfies $\tau_\XX^3=[-2]$ and
  \[ \tau_\XX\{S_1,S_2\}=\{S_2,\twi_1(S_2)[-1]\}.\]
  \item Let $\Upsilon_\XX=\twi_1\circ\twi_2\circ\twi_1\circ[2\XX-3]$ which satisfies $\Upsilon_\XX^2=[X-2]$
  and   \[ \Upsilon_\XX\{S_1,S_2\}=\{ S_2,S_1[\XX-2] \}.\]
  \item The auto-equivalence group $\Aut\D_\XX(A_2)$ is generated by $\twi_i,[1],[\XX]$ and
  sits in the short exact sequence
  \[
    1\to\ST_\XX(A_2)\to\Aut\D_\XX(A_2)\to (\ZZ[1]\oplus\ZZ[\XX])/\ZZ[3\XX-4]\to1.
  \]
\end{itemize}

Applying Theorem~\ref{thm:inducing} and the calculation of $\gldim$ in \cite{Q3}, i.e.
\[
    \gldim\Stab\D_\infty(A_2)=[1/3,\infty),
\]
we have the following

\begin{figure}[h]\centering
\begin{tikzpicture}[yscale=.5,xscale=.5,rotate=0]
%\draw[thin,gray!7,fill=gray!7] plot [smooth,tension=.5] coordinates
%    {(.3+1/3,2)(.3+1,4)(.3+3,8)(.3+9,16)(.3+9,16.1)
%    (-.3-9,16.1)(-.3-9,16)(-.3-3,8)(-.3-1,4)(-.3-1/3,2)};
\foreach \j in {.3,.35,...,4}{
    \draw[dashed,orange!23,fill=orange!11] (0,2^\j) ellipse (3^\j/9-.09 and 3^\j/60);
}
\foreach \j in {4,1}{
    \draw[dashed,fill=orange!7] (0,2^\j) ellipse (3^\j/9-.09 and 3^\j/60);
}
\foreach \j in {1,1.05,...,4}{
    \draw[dashed,orange!34] (0,2^\j) ellipse (3^\j/9-.09 and 3^\j/60);
}
\foreach \j in {1,2,...,4}{
    \draw[orange,dashed,thick] (0,2^\j) ellipse (3^\j/9-.09 and 3^\j/60);
}
\draw[orange](0,2^.3)node{$\bullet$};
\draw[blue,,dotted,thick,fill=cyan!60, opacity=0.2] (0,16) ellipse (11 and 2);
\draw[blue,thick,->,>=stealth] (-3,2)node[below]{$\Re(s)$} to (-3,16.3)node[above]{$\infty$};
\draw[red,font=\scriptsize](0,8)node{$\OStab_s\D_\XX$}(10,4.5)node[]{$\OStab_s\D_\XX$};
\draw[blue,font=\scriptsize](0,16.5)node{$\Stab\D_\infty$}
    (-10,4.5)node{$\Stab\D_\infty$};
\draw(0,0)node[red]{$\OStab\D_\XX$};

\draw[blue, dotted,fill=cyan!60, opacity=0.2] (-10,5) coordinate (P) ellipse (5 and 5);
\foreach \j in {0,10,...,50}{
    \draw[blue,very thin] ($(P)+(210-\j:5)$) to[bend right] ($(P)+(90+\j:5)$);
    \draw[blue,very thin] ($(P)+(90-\j:5)$) to[bend right] ($(P)+(-30+\j:5)$);
    \draw[blue,very thin] ($(P)+(330-\j:5)$) to[bend right] ($(P)+(210+\j:5)$);
}

\draw[orange,font=\scriptsize] (3,2)node[below right]{$\Re(s)\ge2$};
\draw[orange, dashed,fill=orange!50, opacity=0.2] (10,5) coordinate (P) ellipse (5 and 5);
\foreach \j in {120,240,0}{
    \draw[orange] ($(P)+(90+\j:5)$) to[bend left] ($(P)+(210+\j:5)$);
}
\foreach \j in {0,60,...,300}{
    \draw[orange] ($(P)+(90+\j:5)$) to[bend left=60] ($(P)+(150+\j:5)$);
}
\foreach \j in {0,30,...,330}{
    \draw[orange] ($(P)+(90+\j:5)$) to[bend left=75] ($(P)+(120+\j:5)$);
}
\foreach \j in {0,15,...,345}{
    \draw[orange] ($(P)+(90+\j:5)$) to[bend left=90] ($(P)+(105+\j:5)$);
}
\foreach \j in {0}{
    \draw[JungleGreen,thick] ($(P)+(80+\j:5.5)$) edge[bend right,<-,>=stealth] ($(P)+(90+10+\j:5.5)$);}
\foreach \j in {0}{
    \draw[JungleGreen,thick] ($(P)+(80+\j:5.5)+(-20,0)$) edge[bend right,<-,>=stealth] ($(P)+(90+10+\j:5.5)+(-20,0)$);}
\draw[JungleGreen]($(P)+(90:5.5)$)node[above]{$^{\ST_\XX\cong\Br}$}
($(P)+(90:5.5)+(-20,0)$)node[above]{$^{\ZZ_3}$};
\end{tikzpicture}
\caption{The tornado illustration of $\CC\backslash\OStab\D_\XX(A_2)$}\label{fig:tornado}
\end{figure}

\begin{lemma}\cite{BQS,Q3}
\begin{itemize}
\item $\CStab_s\D_\XX(A_2)$ is not empty if and only if $\Re(s)\ge4/3$.
It is connected if and only if $\Re(s)\ge2$.
\item ${\OStab}_s\D_\XX(A_2)$ is not empty if and only if $\Re(s)>4/3$.
It is connected if and only if $\Re(s)>2$ (cf. Figure~\ref{eq:Re(s)<2}).
\item $\CStab_s\D_\XX(A_2)=\OStab_s\D_\XX(A_2)$ if and only $\Re(s)>2$.
\item The fundamental domain for $\CC\backslash\overline{\OStab}_s\D_\XX(A_2)/\Aut$
is $R_s$ in Figure~\ref{fig:A2N},
where the coordinate $z=x+y\mathbf{i}$ satisfies
\begin{gather*}
    e^{\bi \pi z}= Z(S_1) / Z(S_2) ,\\
    \Re(z)=\phi(S_1)-\phi(S_2),
\end{gather*}
and $l_{\pm}$ is given by the equation
$$l_{\pm}=\{ z=x+\mathbf{i} y \mid  x\in(\frac{1}{2},\frac{2}{3}],
        y\pi=\mp\ln(-2\cos x\pi)\}.$$
Moreover, there are two orbitfold points on $\partial R_s$,
one is $(2-s)/2$ with order 2 and the other one is $2/3$ with order 3.
\item The order 3 orbitfold point $\sigma_{G,s}$ solves
 the Gepner equation $\tau_\XX(\sigma)=(-\frac{2}{3}) \cdot \sigma$ (\cite[Thm.~5.10]{Q3}).
\end{itemize}
\end{lemma}

\begin{remark}
In Figure~\ref{fig:tornado} we present
the tornado illustration of $\CC\backslash\OStab\D_\XX(A_2)$.
\end{remark}

\begin{figure}[h]
\begin{tikzpicture}[scale=.9]
\draw[orange, dashed,fill=orange!40, opacity=0.2] (10,5) coordinate (P) ellipse (5 and 5);
\foreach \j in {120,240,0}{
    \draw[orange,fill=white] ($(P)+(90+\j:5)$) to[bend left=32] ($(P)+(210+\j:5)$)
                                to[bend left=-20]  ($(P)+(90+\j:5)$);
}
\foreach \j in {0,60,...,300}{
    \draw[orange,fill=white] ($(P)+(90+\j:5)$) to[bend left=60] ($(P)+(150+\j:5)$)
                                to[bend left=-40] ($(P)+(90+\j:5)$);
}
\foreach \j in {0,30,...,330}{
    \draw[orange,fill=white] ($(P)+(90+\j:5)$) to[bend left=30] ($(P)+(120+\j:5)$)
                                to[bend left=-60] ($(P)+(90+\j:5)$);
}
%\foreach \j in {0,15,...,345}{
%    \draw[orange] ($(P)+(90+\j:5)$) to[bend left=90] ($(P)+(105+\j:5)$);
%}
\end{tikzpicture}
\caption{$\OStab_s \D_\XX(A_2)$ is not connected when $\Re(s)<2$}\label{eq:Re(s)<2}
\end{figure}

\appendix
%=========================================================
\section{Categorification of $q$-deformed root lattices}\label{sec:cat}
%=========================================================

%=========================================================
\subsection{$q$-deformed root lattices}
%=========================================================
Recall the notation $R:=\ZZ[q,q^{-1}]$.
Let $Q$ be an acyclic quiver with vertices $\{1,\dots,n\}$
and $b_{ij}$ be
the number of arrows from $i$ to $j$.

We introduce the $q$-deformed Cartan matrix $A_Q(q)=(a(q)_{ij})$ by
\begin{equation*}
a(q)_{ij}:=\delta_{ij}+q \,\delta_{ji}-(b_{ij}+q \,b_{ji} )
\end{equation*}
where $\delta_{ij}$ is the Kronecker delta.
We note that the specialization $a(1)_{ij}$ at $q=1$ gives the usual generalize Cartan matrix
associated with the underlying Dynkin diagram of $Q$.
The matrix $A_Q(q)$ satisfies the skew symmetric conditions
\begin{equation*}
A_Q(q)^{T}=q\, A_Q(q^{-1}).
\end{equation*}

\begin{definition}
Let $L_Q$ be a free abelian group of rank $n$ with generators
$\alp_1,\dots,\alp_n$ which correspond to vertices $1,\dots,n$
of $Q$:
\begin{equation*}
L_Q:=\bigoplus_{i=1}^n \ZZ \alp_i.
\end{equation*}
We set $L_{Q,R}:=L_Q \otimes_{\ZZ} R$
and define the $q$-deformed bilinear form
\begin{equation*}
(\quad,\quad)_q \colon L_{Q,R} \times L_{Q,R}
\to R
\end{equation*}
by $(\alp_i,\alp_j)_q :=a_{ij}(q)$.
We call $(L_{Q,R},(\,,\,)_q)$ the {\it $q$-deformed root lattice}.
\end{definition}

Corresponding to \emph{simple roots} $\alp_1,\dots,\alp_n$,
we define $R$-linear maps
$r_1^q,\dots,r_n^q:L_{Q,R} \to
L_{Q,R}$ by
\begin{equation*}
r_i^q(\alp):=\alp-(\alp,\alp_i)_q \,\alp_i.
\end{equation*}
Then we can check that
\[
(r_i^q)^{-1}(\alp):=\alp-(\alp_i,\alp)_{q^{-1}} \,\alp_i
\]
by the skew symmetry $q(\alp_i,\alp)_{q^{-1}}=(\alp,\alp_i)$.
The relations of $r_1^q,\dots,r_n^q$ will be described in the next section.
Here we consider the special case $q=1$ and write $r_i:=r_i^{q=1}$.
The group
\[
W_Q:=\left<r_1,\dots,r_n\right>
\]
generated by simple reflections $r_1,\dots,r_n$
is called the {\it Weyl group} and satisfies the relations
\begin{align*}
r_i r_i&=1 \\
r_i r_j &=r_j r_i           \quad \quad \text{if}  \quad  a(1)_{ij}= 0 \\
r_i r_j r_i&=r_j r_i r_j   \,\quad \text{if}  \quad  a(1)_{ij} = -1.
\end{align*}
Note that these relations give the description of $W_Q$ as the Coxeter group.
%=========================================================
\subsection{Artin groups and Hecke algebras}
%=========================================================
\label{sec:Artin}
In this section,
we define the Artin group associated to an acyclic quiver $Q$ and
discuss the representation of it through Hecke algebras.
\begin{definition}[\cite{BSai}]
\label{Artin}
The {\it Artin group $\Br_Q$} is the group generated
by $\bb_1,\bb_2,\dots,\bb_n$ and relations
\begin{align*}
\bb_i \bb_j &=\bb_j \bb_i
\quad \quad \text{if}  \quad  a(1)_{ij}= 0 \\
\bb_i \bb_j \bb_i&=\bb_j \bb_i \bb_j
\quad \text{if}  \quad  a(1)_{ij} = -1.
\end{align*}
\end{definition}

By using reflections $r_1^q,\dots,r_n^q$, we can construct
the representation of $\Br_Q$ on $L_{Q,R}$.

\begin{lemma}
\label{rep}
%Let $W_q:=\left<r_1^q,\dots,r_n^q\right>$ the groups generated by
%$r_1^q,\dots,r_n^q$.
The correspondence of generators
\begin{equation*}
\Br_Q \to \GL(L_{Q,R}),\quad \bb_i^{-1} \mapsto r_i^q
\end{equation*}
gives the representation of $\Br_Q$ on $L_{Q,R}$.
In other words, reflections $r_1^q,\dots,r_n^q$ satisfy
the relations in Definition \ref{Artin}.
\end{lemma}
\begin{proof}
First we show the relation $r_i^q r_j^q r_i^q =r_j^q r_i^qr_j^q$
if $a_{ij}=-1$ for $i \neq j$.
Recall the definition of $a(q)_{ij}$:
\begin{equation*}
a(q)_{ij}=\delta_{ij}+q \,\delta_{ji}-(b_{ij}+q \,b_{ji} ).
\end{equation*}
Since the condition $a(1)_{ij}=-1$ implies $b_{ij}=1$ and $b_{ji}=0$ or vice versa,
we have $a(q)_{ij}=-1$
or $a(q)_{ij}=-q$. In particular,
we have $a(q)_{ij}a(q)_{ji}=q$.
By using this equality, we have the following:
\begin{align*}
r_i^q r_j^q r_i^q(\alp_k)=\alp_k&+\big(a(q)_{ki}a(q)_{ij}-a(q)_{kj}\big)\alp_j \\
&+\big(a(q)_{kj}a(q)_{ji}-a(q)_{ki}\big)\alp_i.
\end{align*}
As the right hand side is an invariant when exchanging $i$ and
$j$, we have $r_i^q r_j^q r_i^q =r_j^q r_i^qr_j^q$.

Similar we have $r_i^q r_j^q=r_j^q r_i^q$
if $a_{ij}=0$.
\end{proof}

If $Q$ is an $A_n$ quiver, the above representation is
known as the reduced Burau representation.

\begin{definition}
The {\it Hecke algebra $H_Q$} is an $R$-algebra generated by
$T_1.\dots,T_n$ with relations
\begin{align*}
(T_i-q)(T_i+1)&=0 \\
T_i T_j &=T_j T_i           \quad \quad \text{if}  \quad  a_{ij}(1)= 0 \\
T_i T_j T_i&=T_j T_i T_j   \quad \text{if}  \quad  a_{ij}(1) = -1.
\end{align*}
\end{definition}

\begin{proposition}
Let $R[\Br_Q]$ be the group ring of $\Br_Q$ over $R$. Then the
representation
\begin{equation*}
R[\Br_Q] \to R[\GL(L_{Q,R})],\quad \bb_i \mapsto r_i^q
\end{equation*}
in Lemma \ref{rep} factors the Hecke algebra $H_Q$:
\begin{equation*}
\xymatrix{
R[\Br_Q] \ar@{|->}[rr] \ar[dr]&&
R[\GL(L_{Q,R})]  & \bb_i^{-1} \ar@{|->}[rr] \ar@{|->}[dr]&&
r_i^q \\
&H_Q \ar[ur] &&& -T_i \ar@{|->}[ur]&.
}
\end{equation*}
\end{proposition}
\begin{proof}
We need to show that
\begin{equation*}
(-r_i^q-q)(-r_i^q+1)=0.
\end{equation*}
We note that $r_i^q(\alp_i)=-q$.
Then we have the following calculation:
\begin{align*}
(-r_i^q-q)(-r_i^q+1)(\alp_k)
&=(r_i^q)^2(\alp_k)+(q-1)r_i^q(\alp_k) -q \alp_k \\
&=r_i^q(\alp_k-a(q)_{ki}\alp_i)+(q-1)(\alp_k-a(q)_{ki}\alp_i)-q \alp_k \\
&=\alp_k -a(q)_{ki}\alp_i +q a(q)_{ki}\alp_i
+(q-1)(\alp_k-a(q)_{ki}\alp_i)-q \alp_k \\
&=0.
\end{align*}
\end{proof}

%=========================================================
\subsection{Grothendieck groups and $q$-deformed root lattices}
%=========================================================
\label{sec:K-group}
In this section, we realize $q$-deformed root lattices as the Grothendieck groups
of derived categories of Calabi-Yau-$\XX$ completions of acyclic quivers.
Recall that $Q$ is an acyclic quiver,
$\bK Q$ be the path algebra of $Q$ and $\Pi_{\XX}(\bK Q)$ its Calabi-Yau-$\XX$ completion.
The Grothendieck $K(\D_{\XX}(Q))$ carries the $R=\ZZ[q^{\pm1}]$-module structure as in \eqref{eq:R-structure}.
Let $S_1,\dots,S_n \in \D_{\XX}(Q)$ be simple modules
of $\Pi_{\XX}(\bK Q)$ corresponding to vertices $\{1,\dots,n\}$ of $Q$.
So we have the following.
\begin{lemma}
\label{K-group}
The Grothendieck group  $K(\D_{\XX}(Q))$ admits a basis $\{[S_i]\}_{i=1}^n$:
\begin{equation*}
K(\D_{\XX}(Q)) \cong \bigoplus_{i=1}^n R[S_i],
\quad [S_i[m+n \XX]]\mapsto (-1)^m q^n\,[S_i].
\end{equation*}
\end{lemma}

Next define the {\it Euler form}
\begin{equation*}
K(\D_{\XX}(Q)) \times K(\D_{\XX}(Q))
\to R
\end{equation*}
by
\begin{equation*}
    \chi(E,F)(q)
    :=\sum_{m,l \in \ZZ}(-1)^m q^l \dim_{\bK } \Hom(E,F[m+l \XX]).
\end{equation*}
Thus we obtain the pair $(K(\D_{\XX}(Q)),\chi)$.
This gives the categorification of  the
corresponding $q$-deformed root lattice as follows.
\begin{proposition}
\label{iso_KL}
The pair $(K(\D_{\XX}(Q)),\chi)$ is isomorphic
to the $q$-deformed root lattice $(L_{Q,R},(\,,\,)_q)$
through the map
\begin{equation*}
    K(\D_{\XX}(Q)) \iso L_{Q,R},\quad [S_i] \mapsto \alp_i.
\end{equation*}
\end{proposition}
\begin{proof}
By Lemma \ref{K-group}, the isomorphism of abelian groups
$K(\D_{Q, \XX}) \iso L_{Q,R}$
is clear. The remaining part is to show that
$\chi(S_i,S_j)=(\alp_i,\alp_j)_q$.
Denote by $\chi_0$ the Euler form on $\DQ=\D^b(\k Q)$.
Corollary~\ref{cor:Lag} implies that
\begin{equation*}
\chi(\hh{L}_Q(E), \hh{L}_Q(F))=\chi_0(E,F)+q \,\chi_0(F,E),
\end{equation*}
where $\hh{L}_Q$ is the Lagrangian immersion and
$E,F \in \DQ$.
For the path algebra $\bK Q$, we can compute
\begin{align*}
&\dim \Hom_{\DQ)}(S_i,S_j)=\delta_{ij}  \\
&\dim \Hom_{\DQ}(S_i,S_j[1])=\dim \Ext^1(S_i,S_j)=b_{ij}.
\end{align*}
Thus we have
$\chi_0(S_i,S_j)=\delta_{ij}-b_{ij}$, and
\begin{align*}
\chi(S_i,S_j)&=\chi_0(S_i,S_j)+q\,\chi_0(S_j,S_i) \\
&=\delta_{ij}-b_{ij}+q(\delta_{ji}-b_{ji})
=(\alp_i,\alp_j)_q.
\end{align*}
\end{proof}

%=========================================================
\subsection{Spherical twists}
%=========================================================
\label{sec:spherical}
Finally, we categorify the action of Hecke algebras defined in Section \ref{sec:Artin}
 through the Seidel-Thomas spherical twists.
An object $S \in \D_{\XX}(Q)$ is called {\it $\XX$-spherical} if
\begin{align*}
\Hom(S,S[i]) =
\begin{cases}
\bK  \quad \text{if} \quad i=0,\XX  \\
\,0 \quad \text{otherwise} .
\end{cases}
\end{align*}

Set
\begin{equation*}
    \Hom^{\ZZ^2}(S,E) \otimes S
    :=\bigoplus_{m,l \in \ZZ}\Hom(S[m+l\XX],E) \otimes S[m+l\XX].
\end{equation*}

\begin{proposition}[\cite{ST}, Proposition 2.10]
For a spherical object $S \in \D_{\XX}(Q)$, there is an
exact auto-equivalence $\twi_S \in \Aut\D_{\XX}(Q)$
defined by the exact triangle
\begin{equation*}
\Hom^{\ZZ^2}(S,E) \otimes S \lto E \lto \twi_S(E)
\end{equation*}
for any object $E \in \D_{\XX}(Q)$.
The inverse functor $\twi_S^{-1} \in \Aut\DXQ$ is given by
\begin{equation*}
\twi_S^{-1}(E) \lto E \lto S \otimes \Hom^{\ZZ^2}(E,S)^\vee .
\end{equation*}
\end{proposition}

Let $S_i=S_i^\XX$ be simple $\qq{\XX}$-modules
corresponding to vertices $\{1,\dots,n\}$ of $Q$.
It is easy to check that
$S_1,\dots,S_n$ are $\XX$-spherical objects.
Thus we can define spherical twists
$\twi_{S_1},\dots,\twi_{S_n} \in \Aut\DXQ$.

The Seidel-Thomas' spherical twist group is defined
to be the subgroup of $\Aut\DXQ$ generated by
spherical twists $\twi_{S_1},\dots,\twi_{S_n}$, i.e.
\begin{equation*}
\ST_\XX(Q) := \left< \twi_{S_1},\dots,\twi_{S_n} \right>.
\end{equation*}

\begin{proposition}[\cite{ST}, Theorem 1.2]
\label{ST_braid}
For the group $\ST_\XX(Q)$, the following relations hold :
\begin{align*}
\twi_{S_i}\twi_{S_j} &= \twi_{S_j}\twi_{S_i} \quad \quad \quad \text{if}  \quad  \chi(S_i,S_j)(1) = 0 \\
\twi_{S_i}\twi_{S_j}\twi_{S_i} &= \twi_{S_j}\twi_{S_i}\twi_{S_j} \,\,\quad \text{if}  \quad  \chi(S_i,S_j)(1) = -1.
\end{align*}
\end{proposition}
Proposition \ref{ST_braid} implies that there is a surjective
group homomorphism
\begin{equation}\label{eq:FF}
    \iota_Q^\XX\colon\Br_Q \to \ST_\XX(Q), \quad \bb_i \mapsto \twi_{S_i}.
\end{equation}

On the Grothendieck  group $K(\D_{\XX}(Q))$, the spherical
twist $\twi_{S_i}$ induces a reflection
$[\twi_{S_i}] \colon K(\D_{\XX}(Q)) \to K(\D_{\XX}(Q))$
given by
\begin{equation*}
[\twi_{S_i}]([E]) = [E] - \chi(S_i,E)(q^{-1})[S_i].
\end{equation*}
Then the inverse of $[\twi_{S_i}]$ is
\[
[\twi_{S_i}]^{-1}([E]) = [E] - \chi(E,S_i)(q)[S_i].
\]
Recall from Proposition \ref{iso_KL} that
$(K(\D_{\XX}(Q)),\chi) \cong (L_{Q,R},(\,,\,)_q)$.
Under this isomorphism, we can identify
the reflection $[\twi_{S_i}]$ on $K(\D_{\XX}(Q))$
with the reflection $r_i^q$ on $L_{Q,R}$ for $i=1,\dots,n$.

\begin{proposition}\label{pp:Hecke}
Through the isomorphism $K(\D_{\XX}(Q)) \cong L_{Q,R}$, the
action of $[\twi_{S_i}]^{-1}$ coincides with the action of $r_i^q$ on
$L_{Q,R}$. In particular, the representation of $\Br_Q$ on
$K(\D_{\XX}(Q))$ defined by $\bb_i^{-1} \mapsto [\twi_{S_i}]^{-1}$ is
equivalent to the representation given in Lemma \ref{rep}.
\end{proposition}
For the $A_n$ quivers, this construction is given
by Khovanov-Seidel (see \cite[Prop.~2.8]{KhS})
as the categorification of  Burau representations.

\begin{remark}
If we consider the usual Calabi-Yau-$N$ completion $\Pi_N(\bK Q)$ for
$N \in \ZZ_{\ge 2}$, then we obtain all results in Section \ref{sec:K-group}
 and Section \ref{sec:spherical}
with the specialization $q=(-1)^N$.
\end{remark}

%=========================================================

%=========================================================
\end{document}